\documentclass{article}
\usepackage{amsthm,amsfonts, amsbsy, amssymb,amsmath,graphicx}
\usepackage{graphics}
\usepackage[english]{babel}
\newtheorem{thm}{Theorem}
\newtheorem{lem}{Lemma}

\newtheorem{crl}{Corollary}

\newtheorem{st}{Statement}
\newtheorem{cj}{Conjecture}

\newcounter{tdfn}
\newenvironment{dfn}
{\vspace{0.15cm}{\bf Definition \arabic{tdfn}.}} {\par
\addtocounter{tdfn}{1}}

\newcounter{trk}
\newenvironment{rk}
{\vspace{0.15cm}{\bf Remark \arabic{trk}.}} {\par
\addtocounter{trk}{1}}

 \def\Z{{\mathbb Z}}
 \def\0{{\mathbbf 0}}
 \def\1{{\mathbbf 1}}

\newcommand{\skcrro}{\raisebox{-0.25\height}{\includegraphics[width=0.5cm]{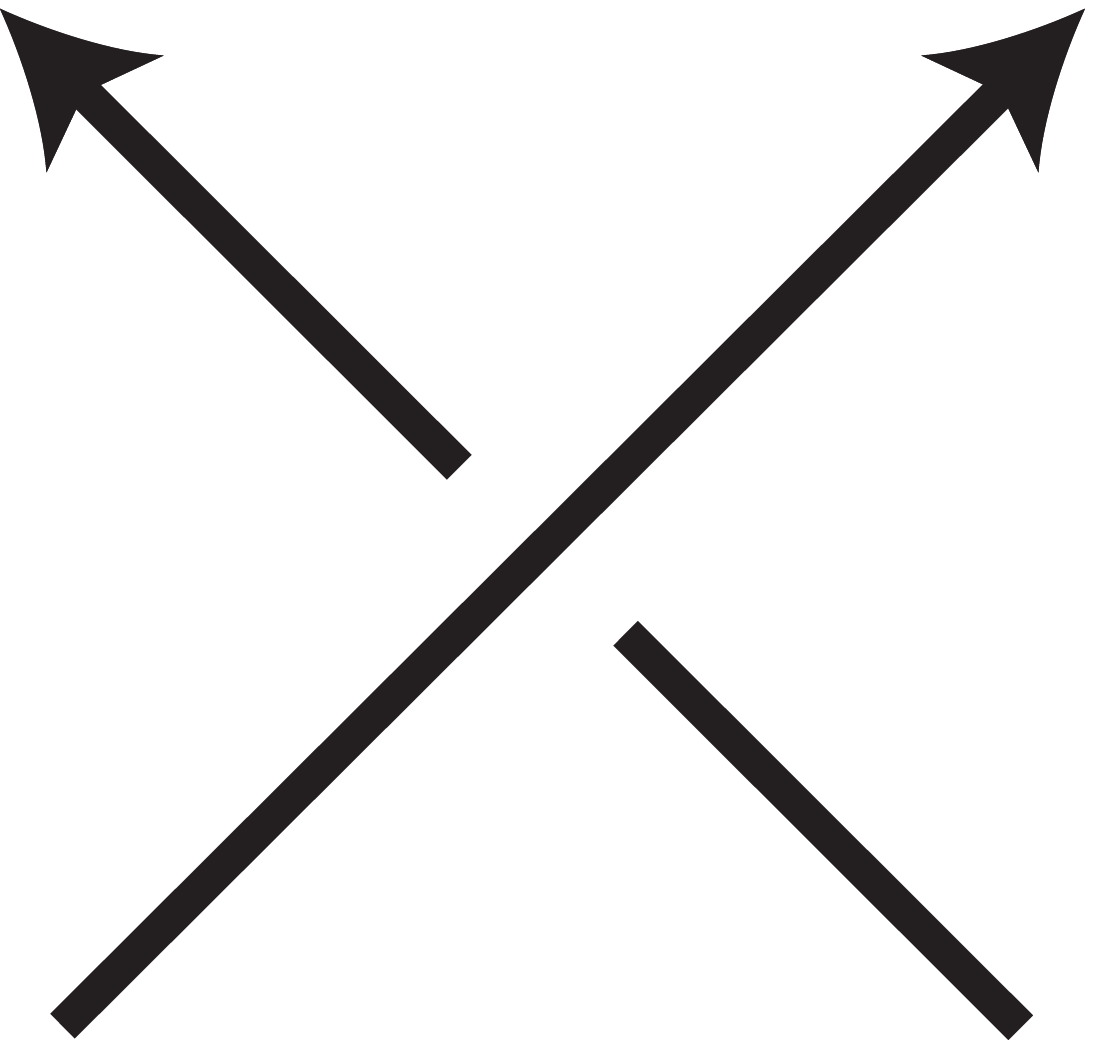}}}
\newcommand{\skcrlo}{\raisebox{-0.25\height}{\includegraphics[width=0.5cm]{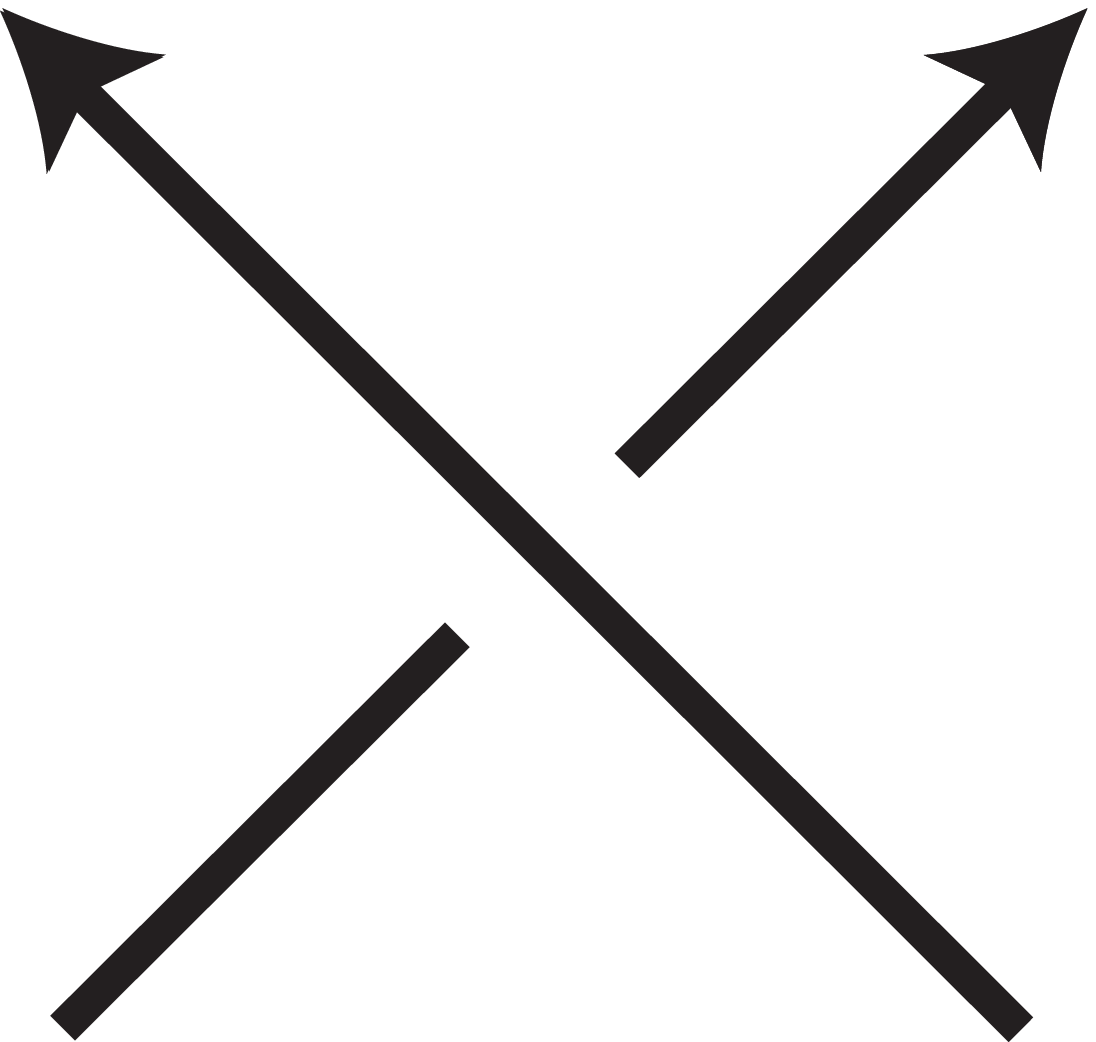}}}

\title{Graphical Constructions for the $sl(3), so(3)$ and $G_{2}$ Invariants  \\ for Virtual Knots, Virtual Braids and Free Knots}

{\author{Louis Hirsch Kauffman\\ Department of Mathematics, Statistics \\ and Computer Science (m/c 249)    \\ 851 South Morgan Street   \\ University of Illinois at Chicago\\
Chicago, Illinois 60607-7045 USA\\
$<$kauffman@uic.edu$>$\\
and\\
Vassily Olegovich Manturov\\
Bauman Moscow State Technical Unversity\\
2nd Baumanskaya St.5/1, Moscow, 105005 Russia\\
and\\
Laboratory of Quantum Topology\\
Chelyabinsk State University\\
Brat'ev Kashirinykh street 129, 
Chelyabinsk 454001, Russia\\
$<$vomanturov@yandex.ru$>$\\}

\begin{document}

\maketitle

\abstract{We construct graph-valued analogues of the Kuperberg
$sl(3)$ and $G_2$ invariants for virtual knots. The restriction of the $sl(3)$ and $G_{2}$
invariants for classical knots coincides with the usual Homflypt
$sl(3)$ invariant and $G_{2}$ invariants. For virtual knots and graphs these invariants provide new graphical information that allows one to prove minimality theorems and to 
construct new invariants for free knots (unoriented and unlabeled Gauss codes taken up to abstract Reidemeister moves).  A novel feature of this approach is that some knots are 
of sufficient complexity that they evaluate themselves in the sense that the invariant is the knot 
itself seen as a combinatorial structure. The paper generalizes these structures to virtual braids and 
discusses the relationship with the original Penrose bracket for graph colorings.}
\smallbreak



\section{Introduction}
This paper studies a generalization to virtual knot theory of
the Kuperberg $sl(3)$ bracket invariant and the Kuperberg $G_{2}$ invariant. 
Kuperberg  discovered bracket state sums that depend upon a reductive graphical procedure
similar to the Kauffman bracket but more complex. The present paper contains new results and results first presented or announced  in \cite{Graph,AC} and puts these results in the more general contexts of minimality, parity, virtual braids and graph coloring. The invariants that we study here apply to virtual knots and links,
flat virtual knots and links, and free knots and links. Free knots and links are essentially the unlabeled Gauss diagrams, taken up to abstract Reidemeister moves, and are fundamental to all aspects of virtual
knot theory.
\bigbreak

In this paper we show that the
Kuperberg brackets can be uniquely defined and generalized to virtual knot
theory via their reductive graphical equations. These equations reduce
to scalars only for the planar graphs from classical knots. For 
virtual knots or free knots, there are unique graphical reductions to linear
combinations of reduced graphs with Laurent polynomial coefficients.
We will call these ``graph polynomials".
The ideal case, sometimes realized, is when the topological object is itself
the invariant, due to irreducibility. When this happens one can point to
combinatorial features of a topological object that must occur in all of
its representatives (first pointed out by Manturov in the context of parity).
These extended Kuperberg brackets specialize to
invariants of free knots and allow us to prove that many free knots are
non-trivial without using the parity restrictions we had been tied to
before.  
\smallbreak

In \cite{Graph, AC} we extended the Kuperberg
combinatorial construction of the quantum $sl(3)$ invariant for the
case of virtual knots. In \cite{KrasnovManturov} the $sl(3)$ invariant is applied to classical knots using virtuality. In this paper, we shall review this construction, the $so(3)$ construction and we shall analyze the corresponding construction for the $G_{2}$ invariant. In the case of the $G_{2}$ construction, we find that this kind of extension works best for free knots. This will be explained in the body of the paper. In speaking of knots in this paper we refer to both knots and links.

In many figures we adopt the usual convention that whenever we give a
picture of a relation, we draw only the changing part of it; outside
the figure drawn, the diagrams are identical. In Figure~\ref{kupbra} we illustrate the basic expansion relations for the $sl(3)$ bracket. Note that two sided faces and quadrilaterals are expanded. The oriented graphs that can appear must, due to the orientations imposed, have an even number of edges per face. Planar graphs with only cubic vertices are seen by the Euler formula to have regions with two, three, four or five faces. Since here three and five sided faces are ruled out, we see that in the plane the expansion of the $sl(3)$ bracket reduces to  scalar. We show that for virtual knots, with their underlying non-planar graph structure, this is no longer the case. The invariant generalizes to virtual knots, but takes values in graphs with Laurent polynomial coefficients. 
\smallbreak

\begin{figure}
     \begin{center}
     \begin{tabular}{c}
     \includegraphics[width=6cm]{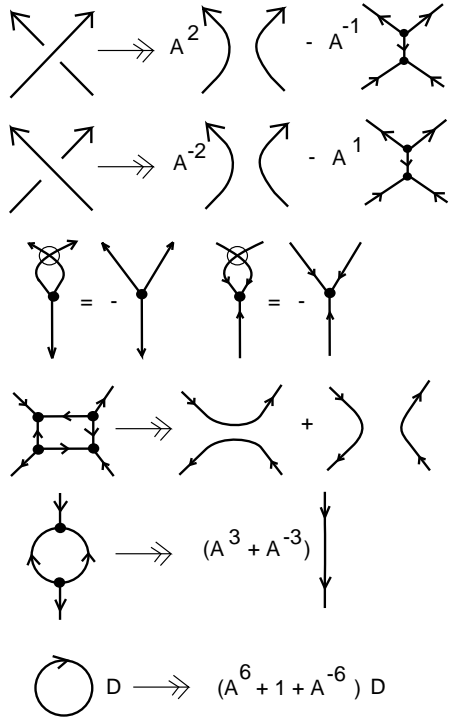}
     \end{tabular}
     \caption{\bf Kuperberg Bracket for $sl(3)$}
     \label{kupbra}
\end{center}
\end{figure}

\begin{figure}
     \begin{center}
     \begin{tabular}{c}
     \includegraphics[width=8cm]{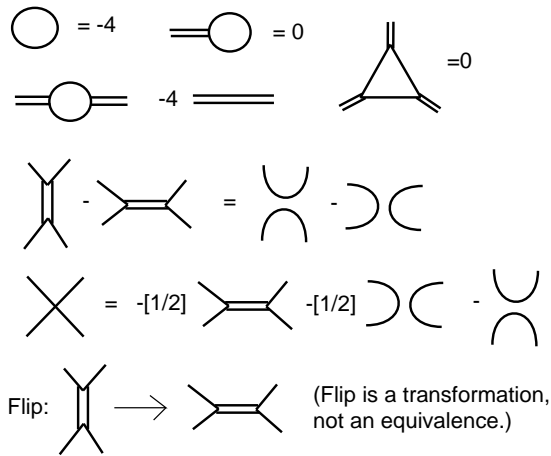}
     \end{tabular}
     \caption{\bf Kuperberg's Relation for $so(3)$}
     \label{kupdbl}
\end{center}
\end{figure}

\begin{figure}
     \begin{center}
     \begin{tabular}{c}
     \includegraphics[width=8cm]{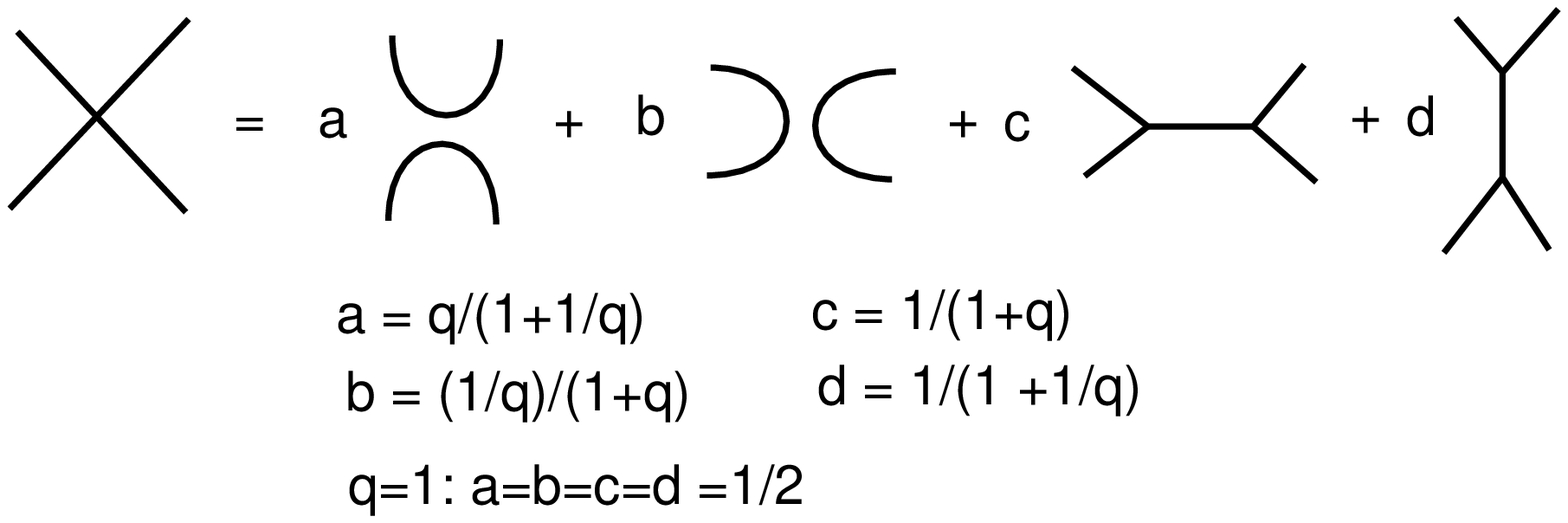}
     \end{tabular}
     \caption{\bf $G_2$ Crossing Expansion}
     \label{crossing}
\end{center}
\end{figure}

\begin{figure}
     \begin{center}
     \begin{tabular}{c}
     \includegraphics[width=8cm]{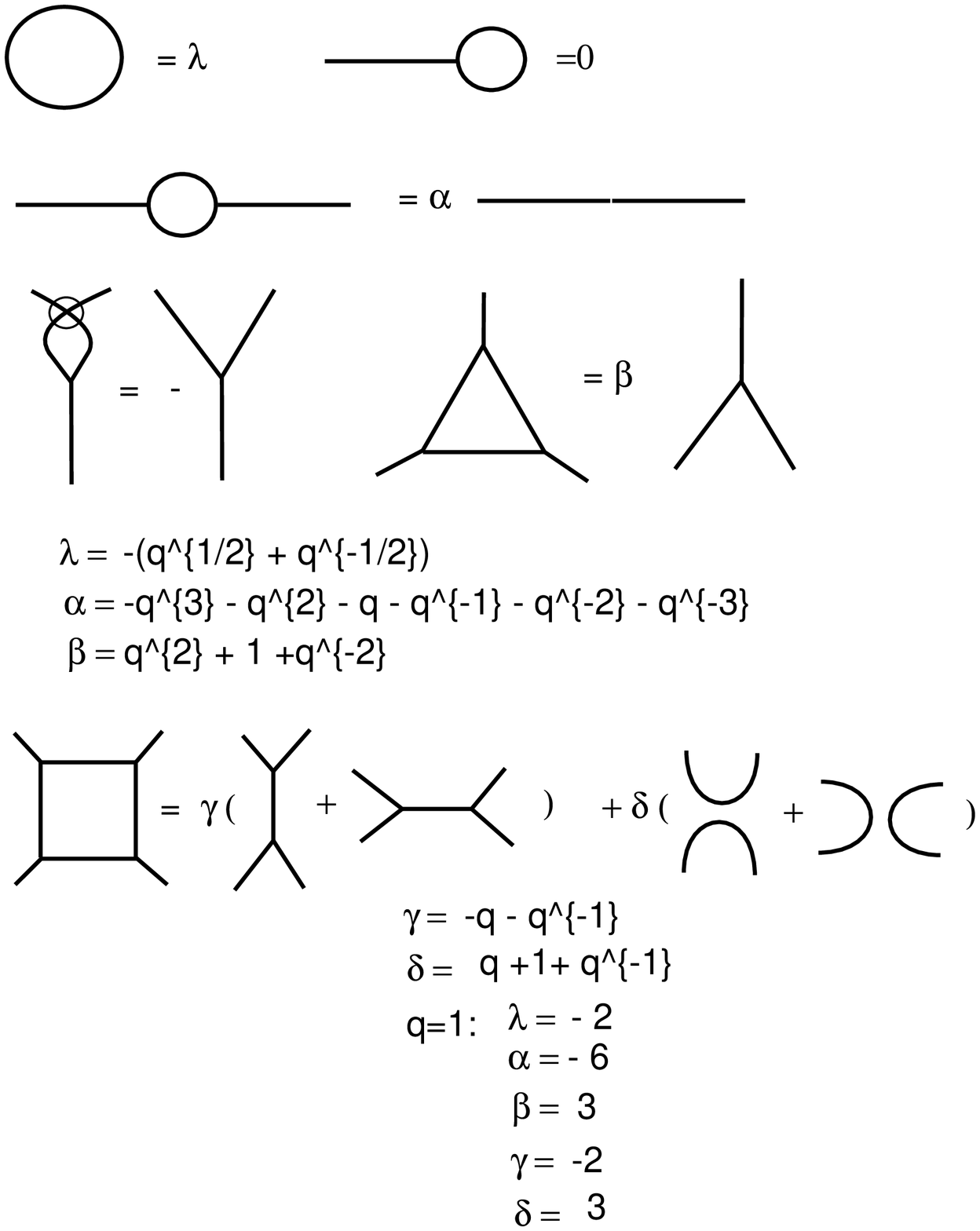}
     \end{tabular}
     \caption{\bf Loop, Triangle and Square Relations for the $G_2$ Bracket}
     \label{looptrisquare}
\end{center}
\end{figure}

\begin{figure}
     \begin{center}
     \begin{tabular}{c}
     \includegraphics[width=8cm]{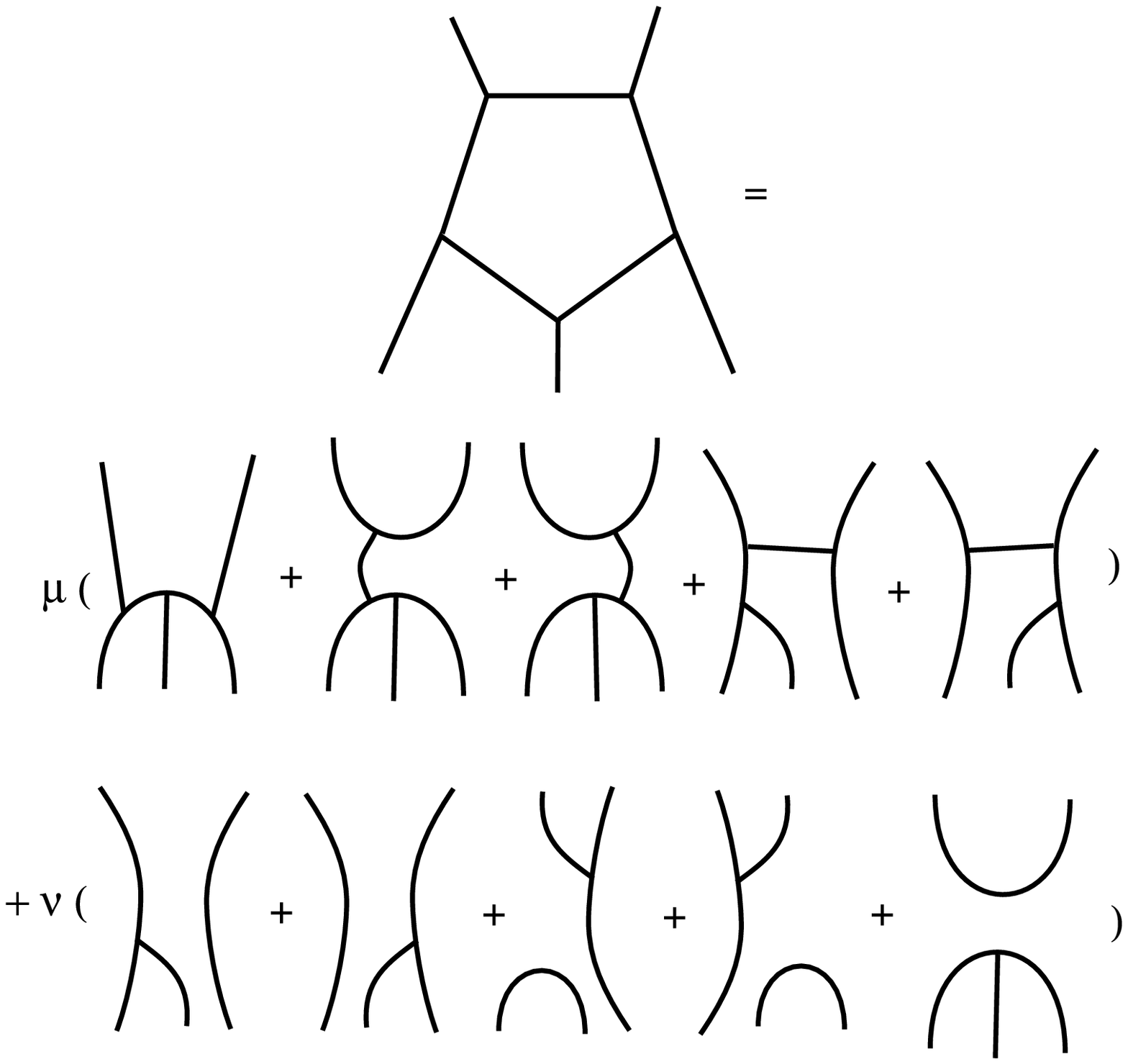}
     \end{tabular}
     \caption{\bf Pentagon Relation for the $G_2$ Bracket}
     \label{pentagon}
\end{center}
\end{figure}

For the case of the $sl(3)$ knot
invariant, one uses the relation shown in Figure~\ref{kupbra}, see \cite{Kup}.
This means that the left (resp., right) picture of (\ref{kupbra})
is resolved to a combination of the upper and lower pictures with
coefficients indicated on the arrows. The
advantage of Kuperberg's approach is that graphs of this sort which
can be drawn on the plane can be easily simplified, by using further
linear relations, to collections of Jordan curves, which in turn,
evaluate to elements from $\Z[A,A^{-1}]$. For planar graphs, these reductions continue all the way to scalars. In the case of non-planar graphs, some  states in the expansion may not contain quadrilaterals or bigons. Such states are irreducible and will receive a possibly non-zero polynomial coefficient in the 
new invariant.  Non-planar resolutions can leave irreducible graphs whose properties reflect the topology of virtual knots and links. The presence of a single such irreducible graph can guarantee the non-triviality and non-classicality of a virtual knot or a free knot. The oriented structure of the states of the $sl(3)$ invariant
makes it possible to sometimes reconstruct the knot from one of its states, but such reconstruction is not guaranteed.
\smallbreak

A triangle is {\it bad} if it has a
vertex with two outgoing edges; a quadrilateral is bad if it has two vertices each 
having two outgoing edges.  The $sl(3)$ invariant yields the minimality of those oriented framed $4$-regular graphs that  have no bad triangles and no bad quadrilaterals. \smallbreak

See Figure~\ref{kupdbl} for the expansion patterns for the $so(3)$ invariant. Here there are single and double lines in the unoriented graphs, and a graph with a maximal number of double lines can be sometimes used to decode the original knot. In the case of this expasion, one obtains an invariant of free knots.
\smallbreak

The main subject of the present paper is an extension of the Kuperberg $G_{2}$ invariant to an invariant of free knots. See Figure~\ref{crossing} for the expansion formula for this invariant, Figure~\ref{looptrisquare} for the loop, triangle and square relations, and Figure~\ref{pentagon} for the pentagon relation. Once again, this invariant can detect free knots when there are irreducible states.
\smallbreak

The present paper is organized as follows.  Section 2 introduces the basics of virtual knot theory including definitions of flat knots and free knots. Section 3 discusses parity and the parity bracket.
Section 4 contains the construction of the $sl(3)$ invariant for virtual knots, flat knots and free knots and the construction of the $G_2$ invariant for flat knots and free knots. Section 5 discusses minimality theorems and modes of detection of knottedness using these invariants. Section 6 gives a small collection of specific examples. Section 7 is a discussion of the Penrose coloring bracket, showing that it is a specialization of the $sl(3)$ invariant. Section 8 puts the $sl(3)$ invariant in the context of virtual braids and virtual Hecke algebra. Section 9 contains closing remarks.
\smallbreak

\noindent {\bf Acknowledgement.} We would like to take this opportunity to thank the 
Mathematisches Forschungsinstitut Oberwolfach for their hospitality and wonderful research atmosphere. A first paper on this subject  was written at the MFO in June of 2012 at a Research in Pairs of the present authors.  The present paper was completed at the MFO in June of 2014 at a second Research in Pairs of the authors.
\smallbreak

The second named author (V.O.M.) was partially supported by Laboratory of Quantum Topology of Chelyabinsk State University (Russian Federation government grant 14.Z50.31.0020) and by grants of the Russian Foundation for Basic Resarch, 13-01-00830,14-01-91161, 14-01-31288.
\bigbreak

\section{Basics of Virtual Knot Theory, Flat Knots and Free Knots}
This section contains a summary of definitions and concepts in virtual knot theory that will be used in the rest of the paper.
\smallbreak

{\it Virtual knot theory} studies the  embeddings of curves in thickened surfaces of arbitrary
genus, up to the addition and removal of empty handles from the surface. See \cite{VKT,DVK}.  
Virtual knots have a special diagrammatic theory, described below,
that makes handling them
very similar to the handling of classical knot diagrams.  \smallbreak  

In the diagrammatic theory of virtual knots one adds 
a {\em virtual crossing} (see Figure~\ref{Figure 1}) that is neither an over-crossing
nor an under-crossing.  A virtual crossing is represented by two crossing segments with a small circle
placed around the crossing point. 
\smallbreak

Moves on virtual diagrams generalize the Reidemeister moves for classical knot and link
diagrams (Figure~\ref{Figure 1}).  The detour move is illustrated in 
Figure~\ref{Figure 2}.  The moves designated by (B) and (C) in Figure~\ref{Figure 1}, taken together, are equivalent to the detour move. Virtual knot and link diagrams that can be connected by a finite 
sequence of classical and detour moves are said to be {\it equivalent} or {\it virtually isotopic}.
A virtual knot is an equivalence class of virtual diagrams under these moves.
\smallbreak

\begin{figure}
     \begin{center}
     \begin{tabular}{c}
     \includegraphics[width=10cm]{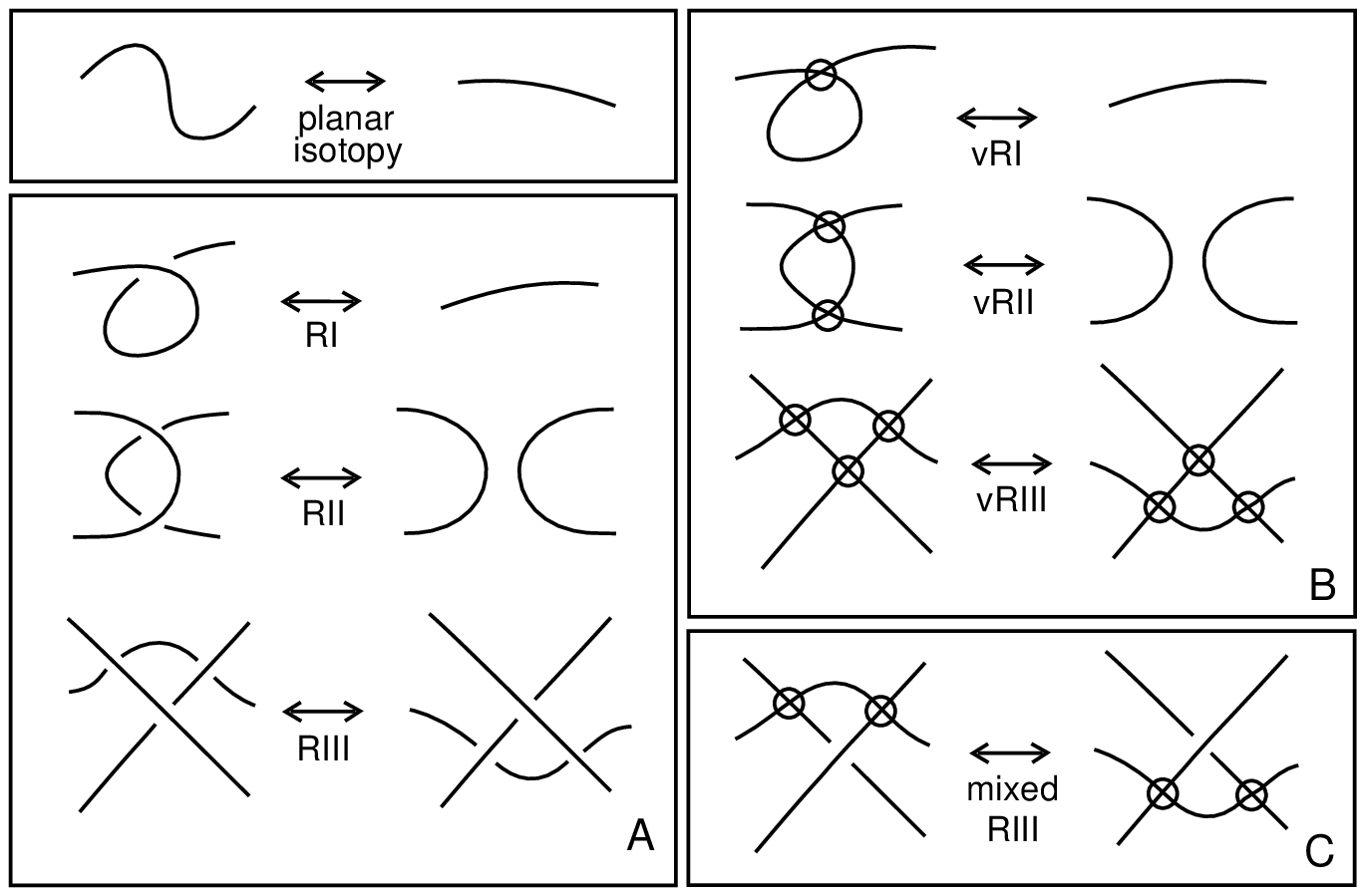}
     \end{tabular}
     \caption{\bf Moves}
     \label{Figure 1}
\end{center}
\end{figure}

\begin{figure}
     \begin{center}
     \begin{tabular}{c}
     \includegraphics[width=10cm]{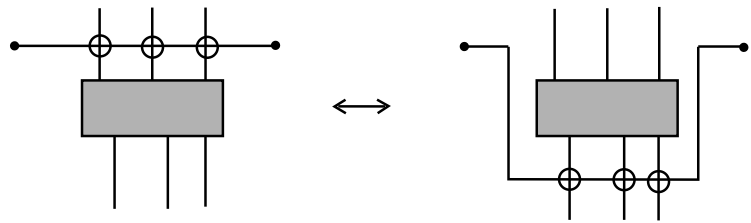}
     \end{tabular}
     \caption{\bf The Detour Move}
     \label{Figure 2}
\end{center}
\end{figure}

Virtual diagrams can be regarded as representatives for oriented Gauss codes \cite{GPV}, \cite{VKT,SVKT} 
(Gauss diagrams). Such codes do not always have planar realizations.   {\it Virtual isotopy is the same as the equivalence relation generated on the collection
of oriented Gauss codes by abstract Reidemeister moves on these codes.}   The reader can see this approach in \cite{DKT,GPV,MB}. It is of interest to know the least number of virtual crossings that can occur in a diagram of a virtual knot or link. If this virtual crossing number is zero, then the link is classical. For some results about estimating virtual crossing number see \cite{DyeKauff,ExtBr,MV} and see the results of Corollaries $3$ and $4$ in Section $3$ of
the present paper. In that section we not only count virtual crossings, we count combinatorial substructures of the diagram that may be unavoidable for a given invariant (in a sense that we specify later in the paper). Note that the breakthrough in the virtual crossing estimate was first achieved in \cite{MB} by applying graphical invariants.
\bigbreak

\begin{figure}
     \begin{center}
     \begin{tabular}{c}
     \includegraphics[width=5cm]{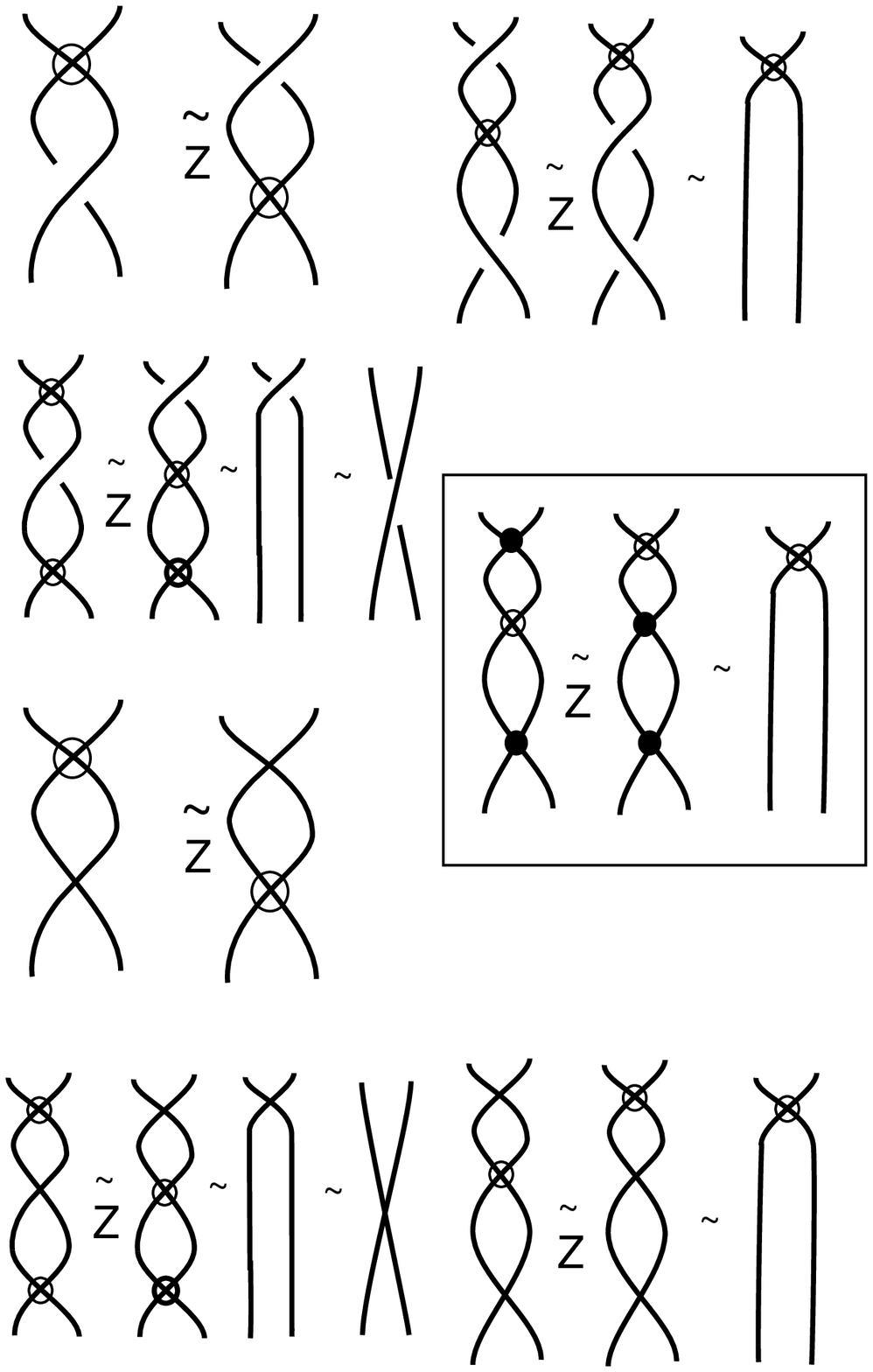}
     \end{tabular}
     \caption{\bf The Z-Move}
     \label{fig2}
\end{center}
\end{figure}

\noindent {\bf Flat Knots and Links.}
Every classical knot diagram can be regarded as a $4$-regular plane graph with extra structure at the 
nodes. Let a {\em flat virtual diagram} be a diagram with {\it virtual crossings} as we have
described them and {\em flat crossings} consisting in undecorated nodes of the $4$-regular plane graph, retaining the cyclic order at a node. Two flat virtual diagrams are {\em equivalent} if
there is a  sequence of generalized flat Reidemeister moves (as illustrated in Figure~\ref{Figure 1}) taking one to the other. A generalized
flat Reidemeister move is any move as shown in Figure~\ref{Figure 1} where one ignores the over or under crossing structure. The moves for flat virtual knots are obtained by taking Figure~\ref{Figure 1} and replacing all the classical crossings by flat (but not virtual) crossings.
In studying flat virtuals the rules for transforming only virtual crossings among themselves and the rules for transforming only
flat crossings among themselves are identical. Detour moves as in part C of Figure~\ref{Figure 1} are available for virtual crossings
with respect to flat crossings and {\it not} the other way around. 
\smallbreak

To each virtual diagram $K$ there is an associated 
flat diagram $F(K)$, obtained by forgetting the extra structure at the classical crossings in $K.$ 
We say that a virtual diagram {\em overlies} a flat diagram if the virtual diagram is obtained from the flat diagram by choosing a crossing type for each flat crossing in the virtual diagram. 
If $K$ and $K'$
are isotopic as virtual diagrams, then $F(K)$ and $F(K')$ are isotopic as flat virtual diagrams. Thus, if we can
show that $F(K)$ is not reducible to a disjoint union of circles, then it will follow that $K$ is a non-trivial  and non-classical virtual link.   
\smallbreak 

\begin{figure}
     \begin{center}
     \begin{tabular}{c}
     \includegraphics[width=8cm]{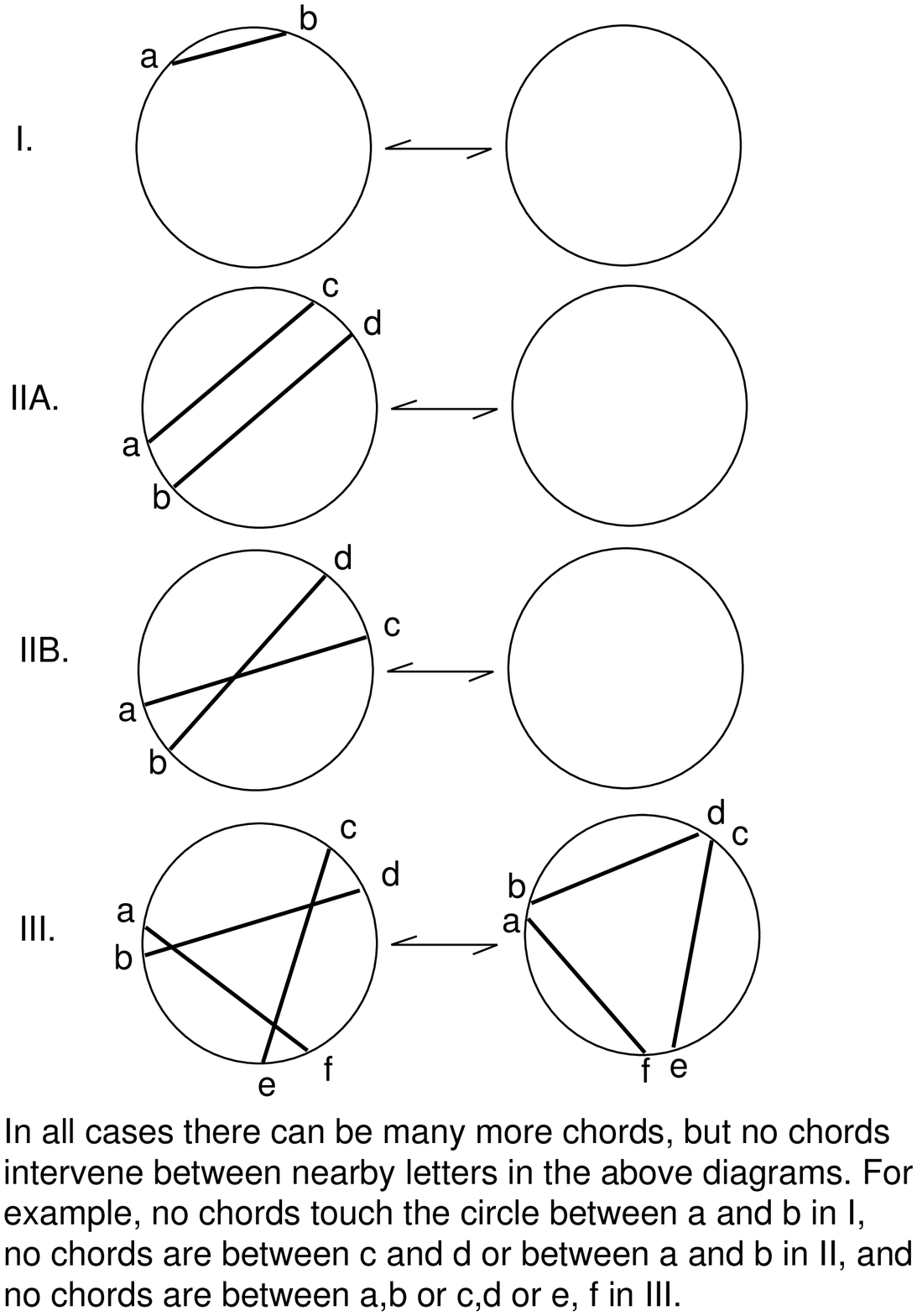}
     \end{tabular}
     \caption{\bf Reidemeister Moves on Chord Diagrams}
     \label{fig3}
\end{center}
\end{figure}

\begin{figure}
     \begin{center}
     \begin{tabular}{c}
     \includegraphics[width=8cm]{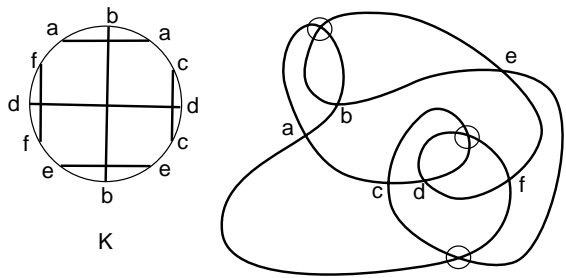}
     \end{tabular}
     \caption{\bf The Free Knot K - Chord Diagram and Virtual Diagram}
     \label{fig4}
\end{center}
\end{figure}

\begin{dfn}
A {\em virtual graph} is a $4-regular$ graph that is immersed in the plane giving a choice of cyclic orders at its nodes.  The edges at the nodes are connected according to the abstract definition of the graph and are embedded into the plane so that they intersect transversely. These intersections are taken as virtual crossings and are subject to the detour move just as in virtual link diagrams.  
We allow circles along with the graphs of any kind in our work with graph theory.
\label{vgraphs}
\end{dfn}

\noindent  {\bf Framed Nodes and Framed Graphs.}
We use the concept of a {\it framed $4$-valent node} where we only specify the pairings of {\it opposite edges} at the node. In the cyclic order, two  edges are said to be opposite if they are paired by skipping one edge as one goes around. If the cyclic order of a node is $[a,b,c,d]$ where these letters label the edges incident to the node, then we say that edges $a$ and  $c$ are {\it opposite}, and that edges
$b$ and $d$ are {\it opposite}.  We can change the cyclic order and keep the opposite relation.
For example, in $[c,b,a,d]$ it is still the case that the opposite pairs are $a,c$ and $b,d.$ A {\it framed $4$-valent graph} is a $4$-valent graph where every node is framed. When we represent a framed $4$-valent graph as an immersion in the plane, we use virtual crossings for the edge-crossings that are artifacts of the immersion and we regard the graph as a virtual graph. For an abstract framed $4$-valent graph, there are neither  classical crossings nor virtual crossings, only the framed nodes and their interconnections.
\smallbreak

\begin{dfn}
A  {\it component} of a framed graph is obtained by taking a walk on the graph so that the walk contains pairs of opposite edges from every node that is met during the walk. That is, in walking, if you enter a node along a given edge, then you exit the node along its opposite edge. Such a walk produces a cycle in the graph and such cycles are called the {\it components} of the framed graph. Since a link diagram or a flat link diagram is a framed graph, we see that the components of this framed graph are identical with the components of the link as identified  by the topologist. A framed graph with one component is said to be {\it unicursal}.  A circle with no nodes is considered as a graph; it is a single component. 
\end{dfn}

Flat knots (and links) are equivalence classes of virtual knots (and links).
Free knots (and links) are equivalence classes of free knots; thus, flat knots
occupy an intermediate position between virtual knots and free knots. It turns
out, however, that flat knots admit a reasonably simple classifications and
in most of the cases they have a unique minimal (w.r.t. flat crossing number)
diagram and in most of the cases there is a descending sequence to the minimal diagram.
For more details, see \cite{HS}. 
\bigbreak

     The reason is that minimal representatives of flat knots are homotopy
classes of curves in a given $2$-surface; virtual knots represent to {\em isotopy}
of curves in $3$-manifolds which is much more complicated than just
isotopy, and flat knots do not have an adequate geometrical representation.
Free knots are almost completely classified; in the present paper we give yet
another invariant which fills some classification gaps.
Thus, we draw our attention mostly to virtual knots and free knots. 
Flat and free knots play a very important role for understanding parities
and other graphical properties of virtual knots.
\bigbreak

\subsection{Free Knots}
\begin{dfn}
{\it Free knots} are equivalence classes of unoriented, unlabeled Gauss codes with the equivalence relation generated by the abstract Reidemeister moves for Gauss codes, as illustrated in 
Figure~\ref{fig3}. Figure~\ref{fig4} illustrates a non-trivial free knot and its corresponding virtual diagram.
Diagrams for free knots are formed just as we form diagrams for virtual knots and flat virtual knots. However, since we do not make any assumptions about cyclic orientation at nodes for the free knot, the virtual diagrams are only defined up to the framing of the classical flat nodes in the diagram in the sense of the discussion above. This means that the free knots can be modified by interchanging a classical flat crossing with an adjacent virtual crossing as shown in  Figure~\ref{fig2}. This interchange is called the 
$Z$-move.  Thus free knots are the same as framed $4$-valent graphs taken up to the flat
Reidemeister moves, and this is the same as flat virtual knots modulo the 
$Z$-move. 
\end{dfn}

Free knots are the most fundamental combinatorial structure underlying virtual knot theory.
If we forget all the structure about a virtual knot except its underlying free knot, then the non-triviality of the free knot (if it is non-trivial) will imply that the virtual knot is also non-trivial. 

Free knots, implicit as they are in the concept of flat virtuals and Gauss codes, were defined by Turaev
\cite{VST,Words} sometime after the concepts of virtual knots and flat virtual knots had been articulated.
It was not at first obvious that free knots could be non-trivial. In fact, Turaev had conjectured that all free knots were trivial in \cite{Words}.  As we shall see in the next section, it can be a subtle matter to prove that a free knot is non-trivial. A nice breakthrough in understanding this subject was made by Manturov when he discovered the role of parity in identifying 
non-trivial free knots. View again Figure~\ref{fig4} and note that every chord has an odd number of intersections with the other chords. As we shall see in the next section, this can be used to show that this is a non-trivial free knot.
\bigbreak

In the $Z$-move one can intechange a crossing with an adjacent virtual crossing even in the category of 
virtual knots and links.  We call virtual knots and links modulo the $Z$-move, $Z$-{\it knots}. At this writing 
we do not know if classical knots and links embed in $Z$-knots.
\bigbreak

\begin{dfn}
We say that a {\it virtualization move} has been performed on a crossing if it is flanked by two virtual crossings. We illustrate this operation in Figure~\ref{fig2} and show that virtualization does not change the equivalence class of a flat diagram under the $Z$-move. This means that any invariant of free knots must be invariant under virtualization.
\end{dfn}

\section{Parity in Knot Theory and Virtual Knot Theory}
This section discusses the use of parity in knot theory, virtual knot theory and particularly in the 
theory of free knots. See \cite{MB,MP,IMN,Projection} for recent work in this area. As we mentioned in the last section, free knots were defined by Turaev
\cite{VST,Words}.
It was not at first obvious that free knots could be non-trivial. A breakthrough in understanding this subject was made by Manturov when he discovered the role of parity in identifying 
non-trivial free knots. This breakthrough was not just in recognizing that virtuality often entails parity. 
We already knew that. The key realization is that through parity it is often possible to identify key combinatorial structures in a diagram or even to identify the whole diagram as such a structure, so that the object whose invariant one wants to compute becomes {\it itself} the invariant. 
It is also important that we localize this non-triviality in concrete crossings. We will illustrate this them with the Manturov Parity Bracket after an initial discussion of the odd writhe of a virtual knot.
\bigbreak

\subsection{The Odd Writhe}
\begin{figure}
     \begin{center}
     \begin{tabular}{c}
     \includegraphics[width=8cm]{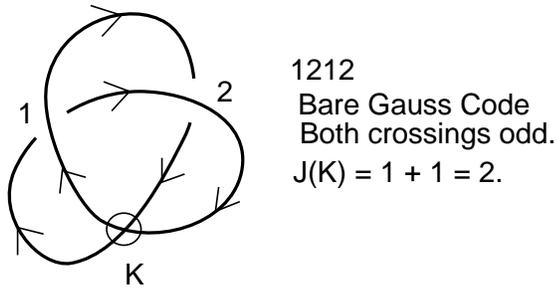}
     \end{tabular}
     \caption{\bf Example of Odd Writhe, J(K).}
     \label{oddwr}
\end{center}
\end{figure}

\begin{figure}
     \begin{center}
     \begin{tabular}{c}
     \includegraphics[width=8cm]{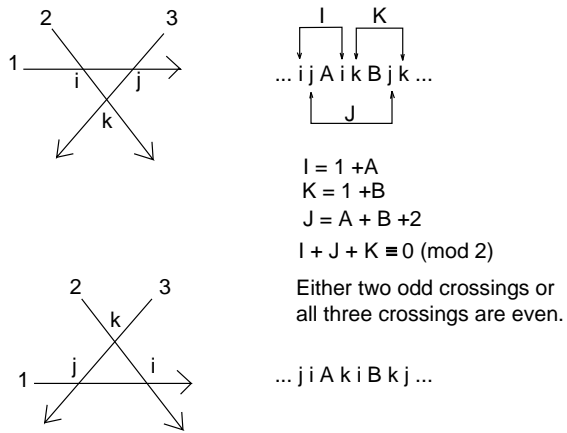}
     \end{tabular}
     \caption{\bf Gaussian Parities at the Third Reidemeister Move.}
     \label{gaussian}
\end{center}
\end{figure}

\noindent {\bf The Odd Writhe J(K).} In Figure~\ref{oddwr}  we show that the virtual knot $K$ is not trivial, not classical and not equivalent to its mirror image by computing its 
{\it odd writhe}  \cite{SL}.  The odd writhe \cite{SL}, $J(K),$ is the sum of the signs of the odd crossings. A crossing
in a knot diagram  is {\it odd} if it flanks an odd number of symbols in the Gauss code of the diagram. 
We call this {\it Gaussian parity} to distinguish it from other parities that can be defined for virtual diagrams. All crossings in a classical knot diagram are even.
Hence classical diagrams have zero odd writhe. In the figure, the flat Gauss code for $K$ is $1212$ with both  crossings odd. Thus we see that $J(K) = 2$ for the knot in the figure.  One proves that 
\begin{enumerate}
\item $J(K)$ is an isotopy invariant of virtual knots, 
\item $J(K^{*}) = -J(K)$ when $K^{*}$ is the mirror image of $K$ (obtained by switching all the crossings), 
\item $J(K) = 0$ when $K$ is isotopic to a classical knot. 
\end{enumerate}
The odd writhe is the simplest application of parity in virtual knot theory.
\bigbreak

Figure~\ref{gaussian} shows why the odd writhe is an invariant of virtual knots by focusing on the third
Reidemeister move and the parity of the three crossings involved in that move. We have indicated three crossings by the lowercase letters $i,j,k$ and have used the capital letters $I,J,K$ to indicate the number of symbols in the Gauss code between two appearances of the corresponding lowercase letter. The diagram in the Figure shows that $I+J+K$ is necessarily an even number. Thus, either two of the crossings at the third move are odd, or all of the crossings are even. Furthermore it follows from the figure that the crossings labeled $i,j,k$ in the before and after figures for the Reidemeister move have the same parity. It is then easy to see, using this information, that $J(K)$ is invariant under virtual isotopy. We will refer to this figure again in the next subsection where we discuss the Parity Bracket.  \smallbreak

\subsection{The Parity Bracket}
In this section we introduce the Manturov Parity Bracket  \cite{MP}.
This is a form of the bracket polynomial defined for virtual knots and for free knots that uses the parity of the crossings. To compute the parity bracket, we first make all the odd crossings into graphical vertices. Then we expand the resulting diagram on the remaining even crossings. The result is a sum of graphs with polynomial coefficients. 
\bigbreak

More precisely, let $K$ be a virtual knot diagram. Let $E(K)$ denote the result of making all the odd crossings in $K$ into graphical nodes as illustrated in  Figure~\ref{pbracket} .
Let $SE(K)$ denote the set of all bracket states of $E(K)$ obtained by smoothing each classical crossing in $E(K)$ in one of the two possible ways. Then we define the {\it parity bracket} 
$$<K>_{P} = (1/d)\Sigma_{S \in SE(K)} A^{i(S)} [S]$$ where $d=-A^2 - A^{-2}$, $i(S)$ denotes the 
product of $A$ or $A^{-1}$ from each smoothing site according to the conventions of  Figure~\ref{pbracket}, and $[S]$ denotes the reduced class of the virtual graph $S.$  The graphs are subject to a reduction move that eliminates bigons as in the second Reidemeister move on a knot diagram as shown in Figure~\ref{pbracket}.  Thus $[S]$ represents the unique minimal representative for the virttual graph $S$ under virtual graph isotopy coupled with the bigon reduction move. A graph that reduces to a circle (the circle is a graph for our purposes) is replaced by the value $d$ above. Thus $<K>_{P}$ is an element of a module  generated by reduced graphs with  coefficients Laurent polynomials in $A.$.
\bigbreak

With the usual bracket polynomial variable $A$, the parity bracket is an invariant of standard virtual knots. With $A=\pm 1$ it is an invariant of flat virtual knots. Even more simply, with $A=1$ and taken modulo two, we have an invariant of flat knots with loop value zero. See Figure~\ref{kishino} for an illustration of the application of the parity bracket to the Kishino diagram illustrated there. The Kishino diagram is notorious for being hard to detect by the usual polynomial invariants such as the Jones polynomial. It is a perfect example of the power of the parity bracket. All the crossings of the Kishino diagram are odd. Thus there is eactly one term in the evaluation of the Kishino diagram by the parity bracket, and this term is the Kishino diagram itself, with its crossings made into graphical nodes. The resulting graph is irreducible and so the Kishino diagram becomes its own invariant. We conclude that this diagram will be found from any isotopic version of the Kishino diagram. This allows strong conclusions about many properties of the diagram. For example, it is easy to check that the least surface on which this diagram can be represented with the given planar cyclic orders at the nodes) is genus two. Thus we conclude that the least genus for a surface representation of the Kishino diagram as a flat knot or virtual knot is two.
\bigbreak

Two virtual knots or links that are related by a $Z$-move have the same standard bracket polynomial. This follows directly from our discussion in the previous section. We would like to analyze the structure of $Z$-moves using the parity
bracket. In order to do this we need a version of the parity bracket that is invariant under the Z-move.
In order to accomplish this, we need to add a corresponding Z-move in the graphical reduction process for the parity bracket. This extra graphical reduction is indicated in Figure~\ref{fig2} where we show a graphical $Z$-move. The reader will note that graphs that are irreducible without the graphical $Z$-move can become reducible if we allow graphical $Z$-moves in the reduction process. For example, the graph associated with the Kishino knot is reducible under graphical $Z$-moves. However, there are examples of  graphs that are not reducible under graphical $Z$-moves and Reidemister two moves.
To obtain invariants of free knots using the parity bracket, we adopt the graphical $Z$-move, since we need this invariance for free knots.  
\bigbreak

\subsection{The Parity Bracket for Free Knots}
The rest of this section is devoted to the theory of the parity bracket for free knots.
Much of this material appears in \cite{Sbornik1,FirstFree} and some of it \cite{Graph}.
Some terminology will be useful: We refer to the {\it graph} of a free knot as the intersection graph of its chord diagram. That is, the nodes of the graph are in $1-1$ correspondence with the chords of the chord diagram, and two nodes are joined by an edge when the two chords intersect in the chord diagram.
 If a graph $\Gamma$ is the intesection graph for the chord diagram of a free knot, we say that the free knot is {\it generated} by  $\Gamma$. The reader should note that intersection graphs have neither loops nor bigons.
 In Figure~\ref{irred} we illustrate the relationship of an intersection graph and its chord diagram.
 \smallbreak
 
 \begin{dfn}
We call a free knot diagram {\it odd} if all of the nodes in its intersection graph have odd degree. (This is the same as saying that every crossing has odd Gaussian parity). We call a simple (no more than one edge between any two nodes) graph  
{\it irreducibly odd} if it is odd  and for every pair of nodes $u$ and $v$ there is a third node $w$ that is 
adjacent to exactly one of $u$ and $v.$ 
\end{dfn}

\begin{figure}
     \begin{center}
     \begin{tabular}{c}
     \includegraphics[width=6cm]{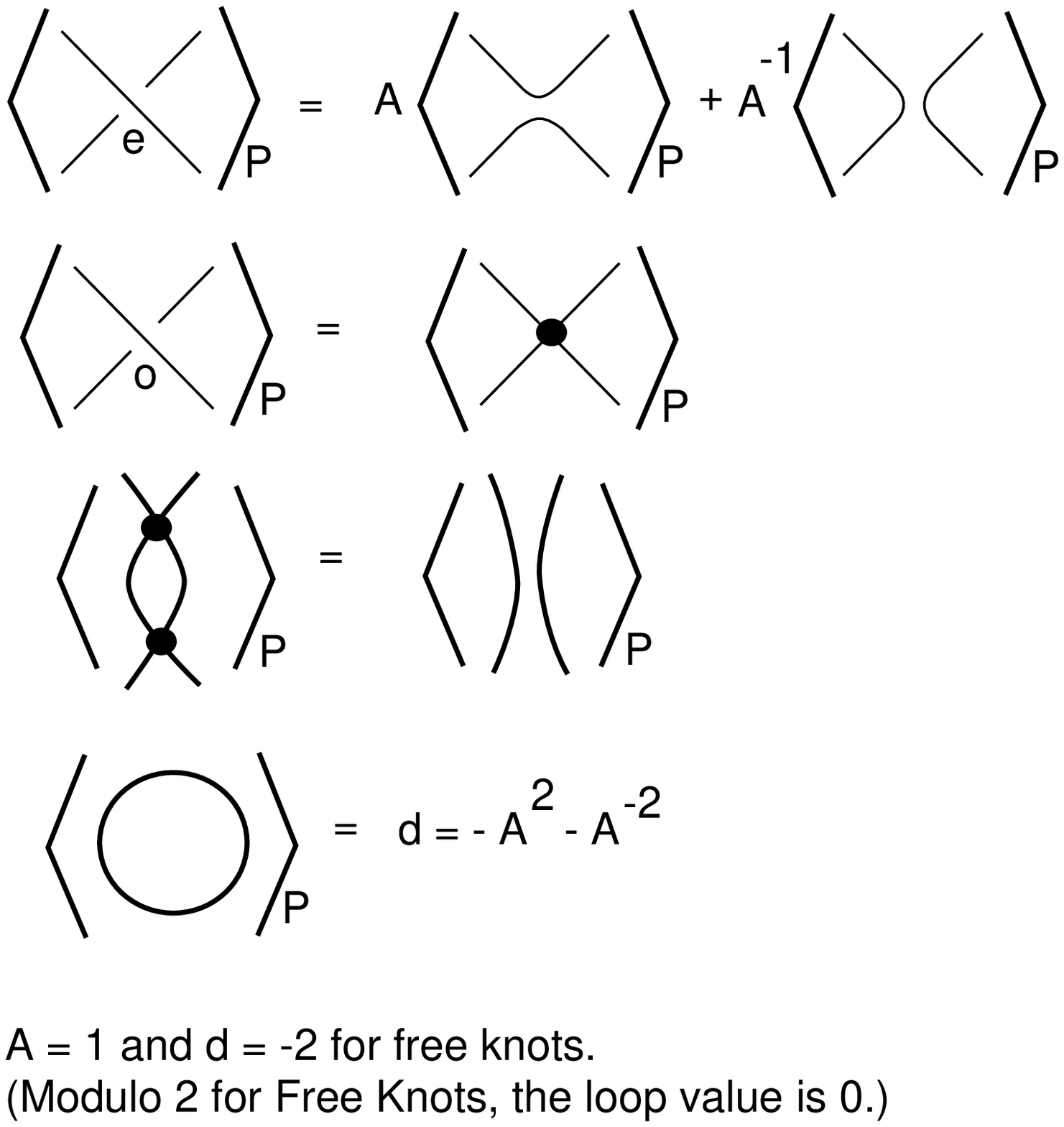}
     \end{tabular}
     \caption{\bf Parity Bracket Expansion}
     \label{pbracket}
\end{center}
\end{figure}

\begin{figure}
     \begin{center}
     \begin{tabular}{c}
     \includegraphics[width=6cm]{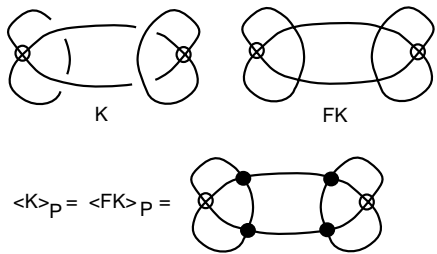}
     \end{tabular}
     \caption{\bf Parity Bracket Detects the Kishino Diagram}
     \label{kishino}
\end{center}
\end{figure}

\begin{figure}
\centering\includegraphics[width=150pt]{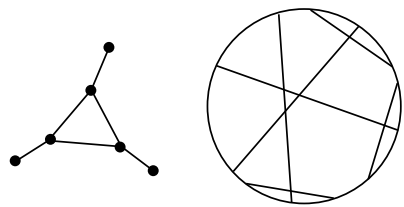} \caption{\bf An
irreducibly odd graph and its chord diagram} \label{irred}
\end{figure}

The simplest example of an irreducibly odd graph is depicted in 
Figure~\ref{irred}.
Assume an irreducibly odd graph $G$ generates a free knot $K$.
It turns out that the representative $G$ of the knot $K$ is indeed {\em
minimal}: any other representative of $K$ has the number of vertices
at least as many as those of $G$.
To this end, we shall describe a powerful invariant \cite{MP} that captures
graphical information about the free knot; but first we should introduce
some notation.
\smallbreak

\newcommand{\ZG}{\Z_{2}{\mathfrak{G}}}
Let ${\mathfrak{G}}$ be the set of all equivalence classes of framed
 graphs with one unicursal component modulo second Reidemeister moves.
Consider the linear space $\ZG$.
Let $G$ be a framed graph, let  $v$ be a vertex of $G$ with four
incident half-edges $a,b,c,d$, s.t.  $a$ is opposite to $c$ and $b$
is opposite to $d$ at $v$.
By {\em smoothing} of $G$ at $v$ we mean any of the two framed
$4$-regular graphs obtained by removing $v$ and repasting the edges as
$a-b$, $c-d$ or as $a-d,b-c$, see Figure~\ref{smooth}.
\smallbreak

\begin{figure}
\centering\includegraphics[width=100pt]{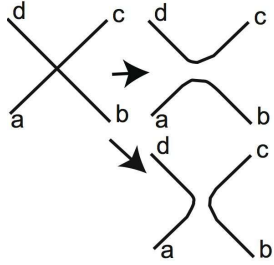} \caption{\bf Two
smoothings of a vertex of for a framed graph} \label{smooth}
\end{figure}

Herewith, the rest of the graph (together with all framings at
vertices except $v$) remains unchanged.
We may then consider further smoothings of $G$ at {\em several}
vertices.
Consider the following sum

\begin{equation}
[G]=\sum_{s\;even,1\; comp} G_{s},
\end{equation}
which is taken over all smoothings in all {\em even} vertices, and
only those summands are taken into account where $G_{s}$ has one
unicursal component. This is the mod-$2$ parity bracket, where the loop value for more than one loop in  state is zero. {\it Note that we do not make non-virtual odd nodes in the diagram into special graphical vertices. We just do not smooth them, but they are still subject to the $Z$-move and to reduction in the form of a second Reidemeister move. Thus in this state summation the remaing nodes take the role of the graphical vertices we previously described for the parity bracket.}
\smallbreak

Thus, if  $G$ has $k$ even vertices, then $[G]$ will contain at most
$2^{k}$ summands, and if all vertices of $G$ are odd, then we shall
have exactly one summand, the graph $G$ itself.
Consider  $[G]$ as an element of $\ZG$. In this case it is evident that if all vertices of $G$ are odd then
 $[G]=G$ (here $G$ is taken as a single graph to be reduce by R-2 moves). 
\smallbreak

\begin{thm}
If $G$ and $G'$ represent the same free knot then in $\ZG$ the
following equality holds: $[G]=[G']$.\label{mainthm}
\end{thm}

Theorem \ref{mainthm} yields the following
\begin{crl}
Let  $G$ be an irreducibly odd framed 4-graph with one unicursal
component. Then any representative $G'$ of the free knot
$K_{G}$,generated by $G$, has a smoothing  $\tilde G$ having the
same number of vertices as $G$.  In particular,  $G$ is a minimal
representative of the free knot $K_{G}$ with respect to the number
of vertices.\label{sld}
\end{crl}

It turns out that elements from $\ZG$ are easily encoded by their
minimal representatives. More precisely, the following lemma holds.
\smallbreak

\begin{lem}
Every $4$-valent framed graph $G$ with one unicursal component
considered as an element of $\ZG$ has a unique irreducible
representative, which can be obtained from $G$ by consequtive
application of second decreasing Reidemeister moves.
\end{lem}

This allows one to recognize elements $\ZG$  easily, which makes the
invariants constructed in the previous subsection digestable.
In particular, the minimality of a framed $4$-regular graph in $\ZG$  is
easily detectable: one should just check all pairs of vertices and
see whether any of them can be cancelled by a second Reidemeister
move (or in $\ZG$ one should also look for free loops). Create all minimal representatives in this way, and then compare them. Indeed, we see that two graphs are R-2-equivalent if their minimal representatives coincide.
\begin{proof}[Proof of the Corollary]
By definition of $[G]$ we have $[G]=G$. Thus if $G'$ generates the
same free knot as $G$ we have $[G']=G$ in $\ZG$.
Consequently, the sum representing $[G']$ in $\mathfrak{G}$ contains
at least one summand which is $a$-equivalent 
($a$-equivalent  means equivalent by Reidemeister two-moves only) to $G$. 
Thus $G'$ has
at least as many vertices as $G$ does.
Moreover, the corresponding smoothing of $G'$ is a diagram, which is
 $a$-equivalent to $G$. One can show that under some (quite natural)
``rigidity'' condition this will yield that one of smoothings of
$G'$ coincides with $G$.
\end{proof}

The use of parity gives this bracket considerable power. For example, consider the diagram in Figure~\ref{fig4}. We see that all the crossings in this free knot are odd. Thus the parity bracket is just the the diagram itself seen as a graph and reduced by $2$-moves. The reader can check that this graph is in fact irreducible in the free sense (see definition above) and the invariant of this free knot is the diagram itself as an irreducible graph. This not only means that this free knot is non-trivial, it also means that any diagram of this free knot will have states that reduce to the diagram in Figure~\ref{fig4}. 
\smallbreak

Parity is clearly an important theme in virtual knot theory and will figure in many future investigations of this subject. The type of construction that we have indicated for the bracket polynomial in this section
can be varied and applied to other invariants. Furthermore the notion of describing a parity for crossings
in a diagram is also susceptible to generalization. For more on this theme the reader should consult
\cite{MP,IMN,Projection,MV} and \cite{SL} for a use of parity for another variant of the bracket polynomial.
\smallbreak

\section{Construction of the Main Invariants}
In this section we shall describe the details of the construction of the $sl(3)$ and $G_{2}$ invariants.
The $so(3)$ invariant has been treated in \cite{AC} and we will not repeat the construction details here.
In all cases of these invariants, the work of verification involves checking consistency of the graphical expansions in situations where two polygons that are subject to expansion share an edge or edges. In our combinatorial approach, such situations must be calculated in detail. In some cases we show these detailed calculations in the paper. In other cases we give the reader instructions for doing the calculation. The purpose of this section is to organize these matters for the reader of the paper.
It is clear to us that a less calculational approach to the theory of these invariants would be preferable, and we will explore that theory in a separate paper.
\bigbreak.

\subsection{The Kuperberg  $sl(3)$ Bracket}

Let ${\cal S}$ be the collection of all trivalent bipartite graphs
with edges oriented from vertices of the first part to vertices of
the second part of the bipartite division of the graph.
Let ${\cal T}=\{t_{1},t_{2},\cdots\}$ be the (infinite) subset of
connected graphs from ${\cal S}$ having neither bigons nor quadrilaterals.
Let ${\cal M}$ be the module $\Z[A,A^{-1}][t_{1},t_{2},\cdots]$ of
formal commutative products of graphs from
${\cal T}$ with coefficients that are Laurent polynomials in one variable $A.$
Disjoint unions of graphs  are treated as products in ${\cal M}$.
Our main invariant will be valued in ${\cal M}$.
\smallbreak

\begin{st}
Figure~\ref{kupbra} shows the reduction moves for the Kuperberg $sl(3)$ bracket. The last three lines of the figure will be called the {\it relations} in that figure.
There exists a unique map $f:{\cal S}\to {\cal M}$ which satisfies
the relations in Figure~\ref{kupbra}. The resulting evaluation yields a topological invariant of virtual links when the first two lines of Figure~\ref{kupbra}  are used to expand the link into a sum of elements of ${\cal S}.$
\end{st}

\begin{proof}
The relations we are going to use to prove the statement are as shown in Figure~\ref{kupbra}.
Note that for the case of planar tangles this map to diagrams modulo relations was constructed
explicitly by Kuperberg \cite{Kup}, and the image was in
$\Z[A,A^{-1}]$. We are going to follow \cite{Kup}, however, in the
non-planar case, the graphs can not be reduced just to collections
of closed curves (in the case of the plane, Jordan curves) and so
later evaluate to polynomials. In fact, irreducible graphs will appear in the non-planar case.
First, we treat every 1-complex with all components being graphs
from ${\cal S}$ and circles: We treat it as the formal product of
these graphs, where each circle evaluates to the factor
$(A^{6} + A^{-6} +1)$.
We note that if a graph $\Gamma$ from ${\cal S}$ has a bigon
or a quadrilateral, then we can use the relations shown in Figure~\ref{kupbra}
(resolution of quadrilaterals, resolution of bigons, loop evaluation)
to reduce it to a smaller graph (or two graphs, then we consider it
as a product).
So, we can proceed with resolving bigons and quadrilaterals until we
are left with a collection of graphs $t_{j}$ and circles; this gives
us an element from ${\cal M}$; once we prove the uniqueness of the resolution, we set the
stage for proving the existence of the invariant. We must carefully check well-definedness and topological invariance.
\smallbreak

In what follows, we shall often omit the letter
$f$ by identifying graphs with their images or intermediate graphs
which appear after some concrete resolutions.
\smallbreak

Our goal is to show that this map $f:{\cal S}\to {\cal M}$ is
well-defined. We shall prove it by induction on the number of graph edges.
{\it The induction base is obvious} and we leave its articulation to the reader.
To perform the induction step, notice that all of Kuperberg's relations
are {\it reductive}: from a graph we get to a collection of simpler
graphs.
\smallbreak

Assume for all graphs with at most $2n$ vertices that the statement
holds. Now, let us take a graph $\Gamma$ from ${\cal S}$ with $2n+2$
vertices. Without loss of generality, we assume this graph is
connected. If it has neither bigon nor quadrilateral, we just take
the graph itself to be its image.
Otherwise we use the relations {\it resolution of bigons} or
{\it resolution of quadrilaterals} as in Figure~\ref{kupbra}  to reduce it to a linear combination of
simpler graphs; we proceed until we have a sum (with Laurent polynomial coefficients) of
(products of) graphs without bigons and quadrilaterals.
\smallbreak

According to the induction hypothesis, for all simpler graphs, there
is a unique map to ${\cal M}$. However, we can apply the relations
in different ways by starting from a given quadrilateral or a bigon.
We will show that the final result does not depend on the bigon or
quadrilateral we start with.
To this end, it suffices to prove that if $\Gamma$ can be resolved
to $\alpha \Gamma_{1}+\beta \Gamma_{2}$ from one bigon
(quadrilateral) and also to $\alpha' \Gamma'_{1}+\beta' \Gamma'_{2}$ from the other one, then
both linear combinations can be resolved further, and will lead to
the same element of ${\cal M}$. This will show that final reductions are unique.
\smallbreak

Whenever two nodes of a quadrilateral coincide, then two edges coincide and it is 
no longer subject to the quadrilateral reduction relation. Thus we assume that quadrilaterals under discussion have distinct nodes.
Note that if two polygons (bigons or
quadrilaterals) share no common vertex then the corresponding two
resolutions can be performed {\em independently} and, hence, the
result of applying them in any order is the same. So, in this case,
$\alpha \Gamma_{1}+\beta \Gamma_{2}$ and $\alpha' \Gamma'_{1}+\beta'
\Gamma'_{2}$ can be resolved to the same linear combination in one
step. By the hypothesis,
$f(\Gamma_{1}),f(\Gamma_{2}),f(\Gamma'_{1}),f(\Gamma'_{2})$ are all
well defined, so, we can simplify the common resoltion for $\alpha
\Gamma_{1}+\beta \Gamma_{2}$ and $\alpha' \Gamma'_{1}+\beta'
\Gamma'_{2}$ to obtain the correct value for $f$ of any of these two
linear combinations, which means that they coincide.
\smallbreak

If two polygons (bigons or quadrilaterals) share a vertex, then
they share an edge because the graph is trivalent.
If  a connected trivalent graph has two different bigons sharing an
edge then the total number of edges of this graph is three, and the
evaluation of this graph in ${\cal T}$ follows from an easy
calculation.
Therefore, let us assume we have a graph $\Gamma$ with an edge shared by a
bigon and a quadrilateral. We can resolve the quadrilateral first,
or we can resolve the bigon first.
The calculation in Figure~\ref{undobigonsquare}
\begin{figure}
\centering\includegraphics[width=250pt]{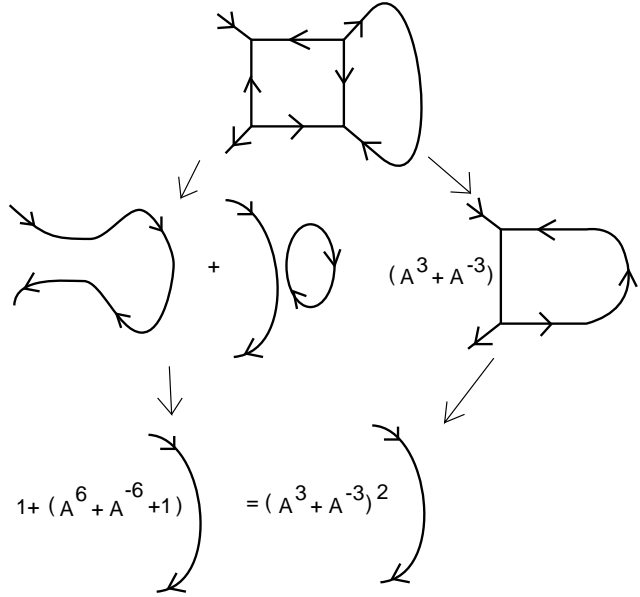}
\caption{\bf Two Ways of Reducing a Quadrilateral and a Bigon}
\label{undobigonsquare}
\end{figure}
shows that after a two-step resolution we get to the same linear
combination.
\smallbreak

A similar situation happens when we deal with two quadrilaterals
sharing an edge, see Figure~\ref{undotwosquares}. Here we have shown
just one particular resolution, but the picture is symmetric, so the
result of the resolution when we start with the right quadrilateral,
will lead us to the same result. See also Figure~\ref{triang}, Figure~\ref{twoquad} and  
Figure~\ref{annular}.
These figures illustrate two other ways in which the edge can be shared. Note that
Figure~\ref{triang} and Figure~\ref{twoquad} and  illustrate a possibly non-planar case and a virtual case, and that we use the abstract graph structure
(with a signed choice of relative order at the trivalent vertex as in Figure~\ref{kupbra} ) in the course of the evaluation. These cases 
cover all the ways that shared edges can occur, as the reader can easily verify. 

\begin{figure}
\centering\includegraphics[width=210pt]{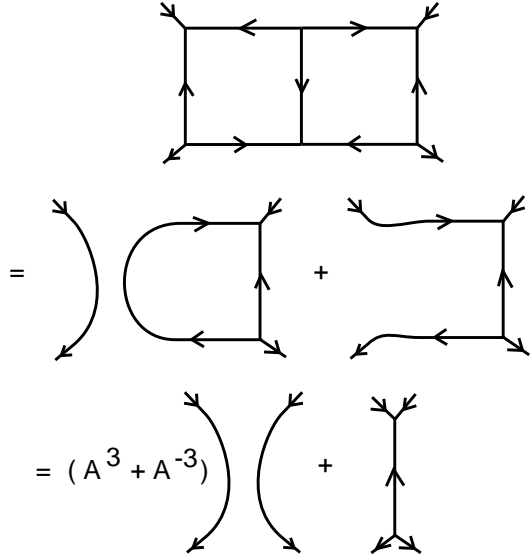}
\caption{\bf Resolving Two Adjacent Squares}
\label{undotwosquares}
\end{figure}

\begin{figure}
\centering\includegraphics[width=200pt]{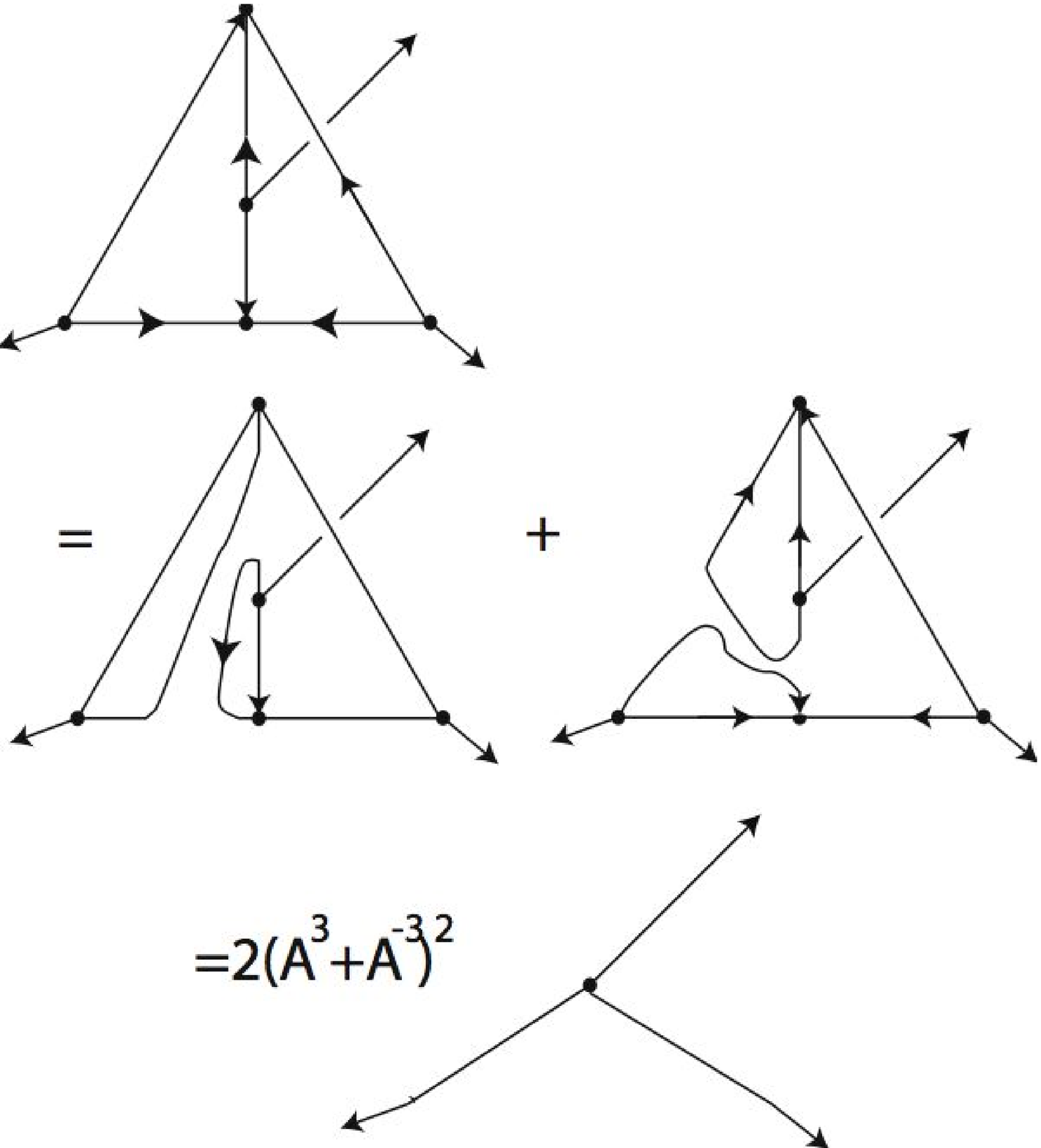}
\caption{\bf Resolving Two Different Adjacent Squares}
 \label{triang}
\end{figure}

\begin{figure}
\centering\includegraphics[width=200pt]{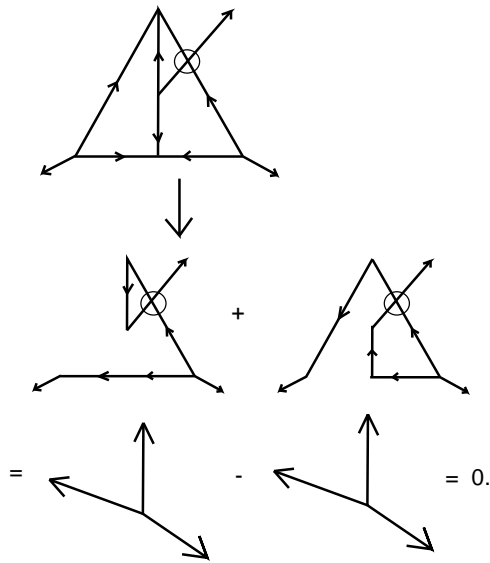}
\caption{\bf Resolving Two Different Adjacent Squares with Virtual Crossing}
 \label{twoquad}
\end{figure}

\begin{figure}
\centering\includegraphics[width=200pt]{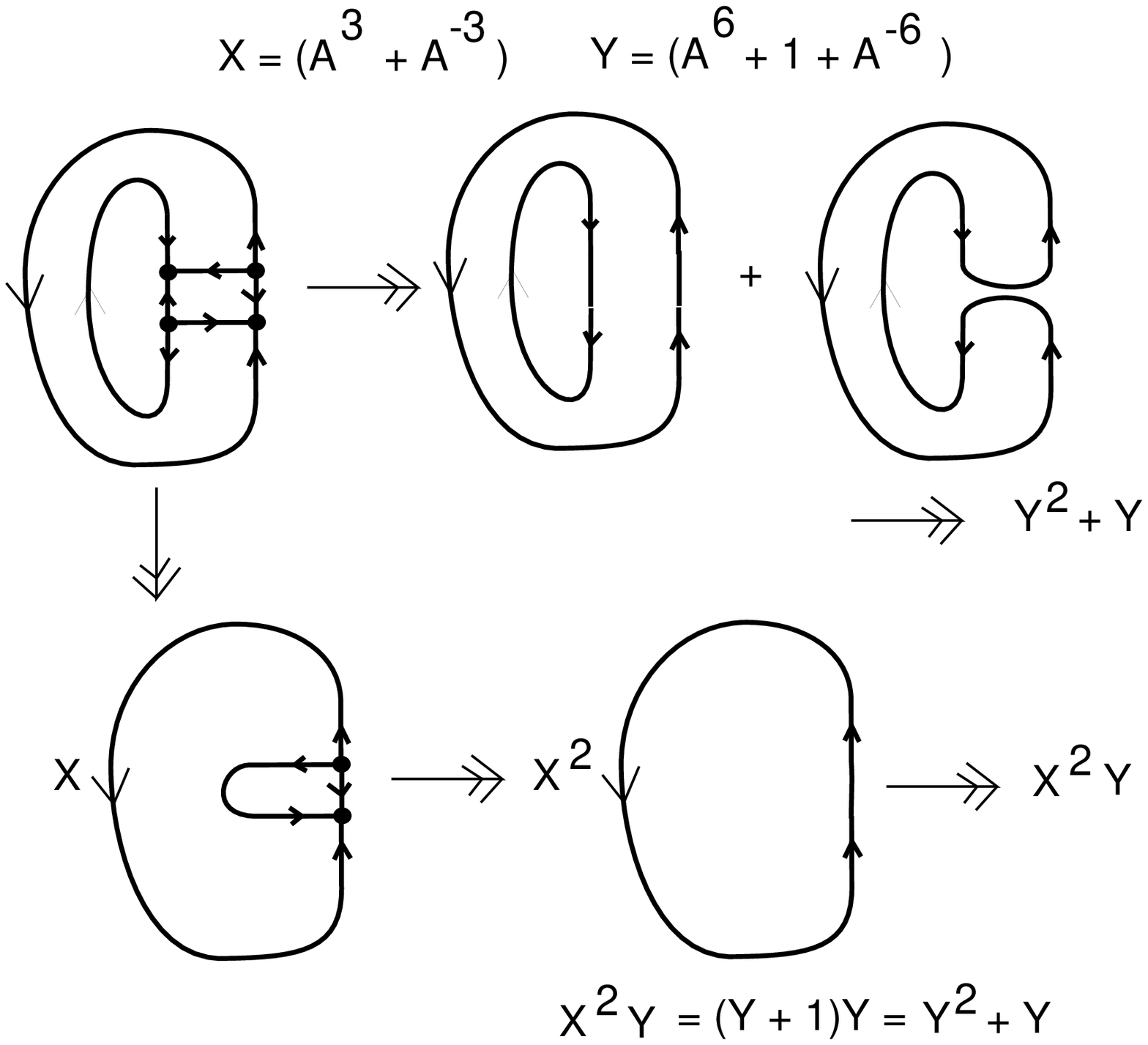}
\caption{\bf Resolving Two Annular Squares} \label{annular}
\end{figure}

Thus, we have performed the induction step and proved the
well-definiteness of the mapping. 
Note that the ideas of the proof are the same as in the classical case; however, we never assumed any planarity of the
graph; we just drew graphs planar whenever possible. Note that the situation in Figure~\ref{triang} is principally non-planar.
The invariance under the $Z$-move  follows from this definition because the graphical pieces into which we expand a crossing, as in Figure~\ref{kupbra}, are, as graphs, symmetric under the interchange produced by the $Z$-move.
\end{proof}

\noindent {\bf Remark.} We can, in the case of flat knots or standard virtual knots represented on surfaces, enhance the invariant by keeping track of the embedding of the graph in the surface and only expanding on bigons and quadrilaterals that bound in the surface. We will not pursue this version of the
invariant here. In undertaking this program we will produce evaluations that are not invariant under the 
$Z$-move for flats or for virtual knots. 
\smallbreak

Now we give a formal description of our main invariant. This evaluation is invariant under the 
$Z$-move. It is defined for virtual knots and links and it specializes to an invariant of free knots.
Let $K$ be an oriented virtual link diagram. With every classical
crossing of $K$, we associate two local states: the {\it oriented} one and
the {\it unoriented} one: the oriented one shown as an oriented smoothing
Figure~\ref{kupbra}, and the unoriented one shown as a connection of two trivalent nodes in 
Figure~\ref{kupbra}. A {\it state} of the diagram is a choice of local state for each crossing in the diagram.
\smallbreak

We define the bracket $[[\cdot]]$ (generalized Kuperberg  $sl(3)$ bracket) as follows. Let $K$ be an
oriented virtual link diagram.
For a state $s$ of a virtual knot diagram $K$, we define the weight
of the state as the coefficient of the corresponding graph according
to the Kuberberg relations Figure~\ref{kupbra}. More precisely, the weight of a
state is the product of weights of all crossings, where a weight of
a positive crossing is $A^{ 2wr}$ for the oriented resolution and
$-A^{- wr}$ for the unoriented resolution, $wr$ stands for the writhe
number (the oriented sign) of the crossing.

Set
\begin{equation}
[[K]] = \sum_{s} w(K_{s})\cdot f(K_{s}) \in {\cal M},
\end{equation}

where $w(s)$ is the weight of the state.

\begin{thm}
For a given virtual diagram $K,$ the normalized bracket $(A^{-8wr(K)})[[K]]$ is invariant under all Reidemeister moves
and the $Z$-move. Recall that $wr(K)$ denotes the writhe obtained by summing the signs of all the classical crossings in the corresponding diagram.
\end{thm}

\begin{proof}
The invariance proof under Reidemeister moves repeats that of Kuperberg.
Note that the writhe behaviour is a consequence of the relations in Figure~\ref{kupbra}.
The only thing we require is that the Kuperberg relations
(summarized in Figure~\ref{kupbra}) can be applied to yield unique reduced graph polynomials.
 The discussion preceding the proof, proving Statement $1$,  handles this issue.
\end{proof}

From the definition of $[[K]]$ we have the following
\begin{crl}
If $[[K]]$ does not belong to $\Z[[A,A^{-1}]]\subset {\cal M}$ then
the knot $[[K]]$ is not classical. That is, if non-trivial graphs appear in the evaluation, then the knot is not classical.
\end{crl}

Recalling that a free link is an equivalence of virtual knots modulo
$Z$-moves and crossing switches and taking into account that
the skein relations of Figure~\ref{kupbra} for $[[\cdot]]$ for $\skcrro$
and $\skcrlo$ are the same when specifying $A=1$, we get the
following

\begin{crl}
$[[K]]_{A=1}$ and $[[K]]_{A= -1}$ are invariants of free links.
\end{crl}

By the unoriented state $K_{us}$ of virtual knot diagram (resp.,
free knot diagram) $K$ we mean the state of $K$ where all crossings
are resolved in the unoriented fashion where the crossing is replaced by two connected trivalent nodes. {\bf Notation:}
$K_{us}$. Note that $K_{us}$ is treated as a graph.

\begin{crl} Assume for a
virtual knot (or free knot) $K$ with $n$ classical crossings the
graph $K_{us}$ has neither bigons nor quadrilaterals. Then every
knot $K'$ equivalent to $K$ has a state $s$ such that $K'_{s}$
contains $K_{us}$ as a subgraph. This state can be treated as an element of ${\cal M}$.
In particular, $K$ is minimal.
 \label{crl1}
\end{crl}

Note that the coincidence of $K_{us}$ and $K'_{us}$ does not
guarantee the coincidence of $K$ and $K'$. For example, if $K$ and
$K'$ differ by a third unoriented Reidemeister move, then, of
course, $[[K]]=[[K']]$. The corresponding resolutions $K_{us}$ and
$K'_{us}$ will coincide (they will have a hexagon inside).
 
\begin{crl}
Let $K$ be a four-valent framed graph with $n$ crossings and with
girth number at least five.
Then the hypothesis of Corollary \ref{crl1} holds. \label{crl2}
\end{crl}

So, this proves the minimality of a large class of framed
four-valent graphs regarded as free knots: all graphs having girth
$\ge 5$ and many other knots. For example, consider the free knot $K_{n}$
whose Gauss diagram is the $n$-gon, $n>6$: it consists of $n$ chords
where $i$-th chord is linked with exactly two chords, those having
numbers $i-1$ and $i+1$ (the numbers are taken modulo $n$). Then $K_{n}$
satisfies the condition of \ref{crl1} and, hence, is minimal in a
strong sense.

Note that the triviality of such $n$-gons as free knots was proven
only for $n\le 6$.

\begin{rk}
The above argument works for links and tangles as well as knots.
\end{rk}

From the construction of $[[\cdot]]$ we get the following corollary.
\begin{crl}
Let $K$ be a virtual (resp., flat) knot, and let
$\Gamma_{1}\cdots\Gamma_{k}$ be a product of irreducible
graphs which appear as a summand in $[[K]]$ (resp.,
$[[K]]|_{A=1}$) with a non-zero coefficient.
 Then the minimal virtual crossing number of $K$ is greater than or
equal to the sum of crossing numbers of graphs: $cr(\Gamma_{1})+\cdots + cr(\Gamma_{k})$ and the underlying genus of $K$ is
less than or equal to the
sum of genera $g(\Gamma_{1})+\cdots + g(\Gamma_{k})$
(in virtual or free knot category).
\end{crl}

The above corollary easily allows one to reprove the theorem first proved in \cite{MV},
that the number of virtual crossings of a virtual knot grows quadratically with respect
to the number of classical knots for some families of graphs. In \cite{MV}, it was
done by using the parity bracket. Now, we can do the same by using $$Free[[K]] = [[K]]|_{A=1}.$$
With this invariant one can easily construct infinite series of trivalent bipartite graphs which serve as $K_{us}$ for some sequence of knots $K_{n}$ and such
that the minimal crossing number for
these graphs grows quadratically with respect to the number of crossings. Recalling that
the number of vertices comes from the number of classical crossings of $K_{n}$, we
get the desired result. 
\smallbreak

\subsection{The Generalized Kuperberg  $G_{2}$ Bracket}
The analysis for the Kuperberg $G_{2}$ invariant follows the same lines as our analysis for the 
$sl(3)$ invariant. The expansion rules for this invariant are given in Figure~\ref{crossing}, Figure~\ref{looptrisquare} and Figure~\ref{pentagon}. It is implicit in Kuperberg's work that collisions of bigons, triangles, quadrilaterals and pentagons that do not involve virtual crossings lead to consistent results for the $q$ variable indicated in these figures. We do not repeat these verifications here. However, when we make the analysis with collisions involving virtual crossings we find that it is in fact necessary to restrict to the case of flat virtual knots or free knots in order to obtain consistency. Thus we take $q=1$ for our $G_{2}$ invariant. The reasons for restricting to $q=1$ becomes apparent from the verifications but we will not show the details here.
\bigbreak

Along with restricting to $q=1$ we have to explain how to handle the reduction calculus for virtual graph diagrams in the plane. A graph diagram in the plane may have (flat) $4$-regular crossings and trivalent crossings. Each such crossing is endowed with a cyclic order by its embeddding in the plane.
For flat virtual knots we do not allow the cyclic order at a crossing to change. This means that for flat knots, we do not consider the reduction of collisions of polygons that involve virtual crossings that cannot be isotoped away from the polygon. For example, view Figure~\ref{fourvirtfive} where we have a four-sided region that can be reduced and a five-sided region that has a virtual intersecting arc. We can for free knots perform the reduction of the four-sided region, but there is no way to reduce the five-sided region. Consequently there is no need for the free-knot $G_{2}$ invariant to examine, as is shown in the figure, what will happen when we change a local cyclic order and reduce the five-sided region. For the invariant of free knots that we are about to describe, we need that the results of expansion from the two local drawings in Figure~\ref{fourvirtfive} give the same results. We will discuss that below.
\bigbreak

The reader should now see that the $G_2$ invariant has two distinct flavors, one for flat virtual knots an links and one for free knots and links. The free knot version involves more checking for collisions than the flat knot version. One can verify that for those collisions that do not involve virtual crossings such as illustrated in Figure~\ref{fivefivedir},    Figure~\ref{fivefour},   Figure~\ref{fivefourend} the expansions are independent of which polygon is first expanded. This provide the proof of validity for the flat knot invariant. The flat knot invariant is of interest, but it should be pointed out that flat virtual knots and links are fully classified via \cite{Kadokami,HS}. %
\bigbreak

\noindent{\bf The $G_2$ Invariant for free knots.}
For free knots, we allow the $Z$-move as part of the equivalence relation, and hence allow the corresponding change in cyclic orde at a trivalent vertex. For the trivalent vertices we use the rule indicated in Figure~\ref{looptrisquare} that shows a change of sign if there is a transposition applied to the cyclic order at a trivalent vertex.  That is, the invariant will be valued in an algebra of graphs with integer coefficients where there is an equivalence relation on the graphs generated by the reduction and transformation rules indicated in the figures such as the above and the ones we now discuss. We can formalize this just as we have already done for $sl(3).$
Furthermore, {\it we only allow reductions on bigons, triangles, quadrilaterals and pentagons that have an interior that is clear of any other aspects of the graph.} This principle is illustrated in Figure~\ref{quad} where we start with a quadrilateral that has an arc crossing it virtually. We change a cyclic order at one of its trivalent vertices, move the arc by a detour move, and then expand the quadrilateral.
\bigbreak

\begin{dfn}
We let $\{\{K\}\}_{Free}$ and $\{\{K\}\}_{Flat}$ denote the $G_2$ invariants for $K$ free and flat respectively. 
\end{dfn}

\begin{dfn}
By a {\em leading state} of a framed $4$-regular graph $G$ on $n$ vertices
(considered as a representative of a free link or a flat link) we mean a state where each
of the $n$ vertices is resolved with an additional edge. 
\end{dfn}

Thus, every $4$-regular graph $G$ has $4^{n}$ states in the expansion
of $\{\{G\}\}_{Free}$ or $\{\{G\}\}_{Flat}$, $2^{n}$ of which are leading
Note that all these graphs $G_{s}$ corresponding to various leading states
$S$, appear in the expansion of $\{\{G\}\}_{Free}$ (resp., $\{\{G\}\}_{Flat}$)
with the coefficient $1/2^{n}$.
\bigbreak

\begin{figure}
     \begin{center}
     \begin{tabular}{c}
     \includegraphics[width=8cm]{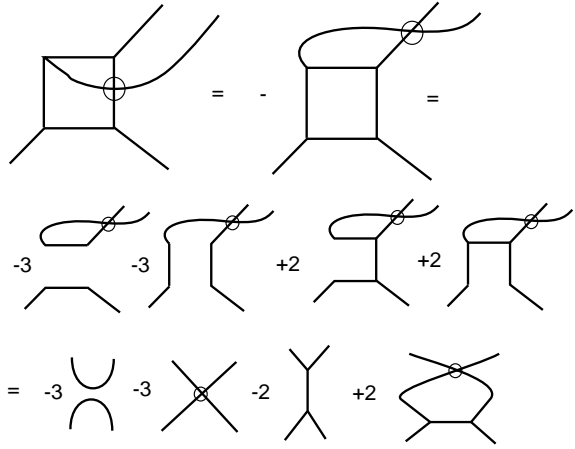}
     \end{tabular}
     \caption{\bf Evaluating a Virtual Quad}
     \label{quad}
\end{center}
\end{figure}

In Figure~\ref{twist} we show that the resulting flat virtual evaluation is invariant under the $Z$-move.
Hence our $G_2$ evaluation is defined on free knots. Note that we use crucially the sign changes from the reordering of cyclic vertices in this derivation. In both Figure~\ref{quad} and Figure~\ref{twist}, we are using the $q=1$ values for the expansions. We leave it to the reader to see that this invariant will not work except at $q=1.$ Figure~\ref{flink} shows the expansion and evaluation of the $G_2$ invariant for the simplest free link. Figure~\ref{fivefivedir},    Figure~\ref{fivefour} and Figure~\ref{fivefourend} illustrate the consistency of polygon collisions where there are no local virtual crossings. 
Figure~\ref{fourvirtfour},    Figure~\ref{fourvirtfive},   Figure~\ref{fivevirtfive},   Figure~\ref{fivevirtfiveinsert}
and Figure~\ref{fivexp} illustrate aspects of collision consistency for four and five sided polygons. We have not given all details, but we have specified in this section how to make these expansions and compare them. For example, in  Figure~\ref{fourvirtfive} we point out that due to the change of cyclic order there should be a sign change when expanding on four versus five in this configuration. We assert that this does indeed happen and leave out the long details.  In  Figure~\ref{fivevirtfive} we indicate a sign change due to cyclic order and that the righthand side of the first equality can be transformed by a $180$ degree rotation (keeping tangle ends fixed) to the lef-hand side. This means that if we expand on just the left-hand side, the resulting expansion should go into the negative of itself under this tangle rotation. Figure~\ref{fivevirtfiveinsert} shows explicitly how certain free trivalent virtual tangles are transformed under the $180$ degree rotation. The arrows on the tangle-boxes at the top of the figure are placed to remind the reader that the tangle-box turns through $180$ degrees, with the ends of the tangle (at its four corners) remaining fixed. Finally Figure~\ref{fivexp}  shows the end-result of the expansion of the lefthand side of  Figure~\ref{fivevirtfive}. It is then easy for the reader to check that this expansion does have the requisite property under the $180$ degree rotation. This completes our sketch of the proof of the validity of the $G_2$ invariant for free knots and links.
\bigbreak

\begin{figure}
     \begin{center}
     \begin{tabular}{c}
     \includegraphics[width=8cm]{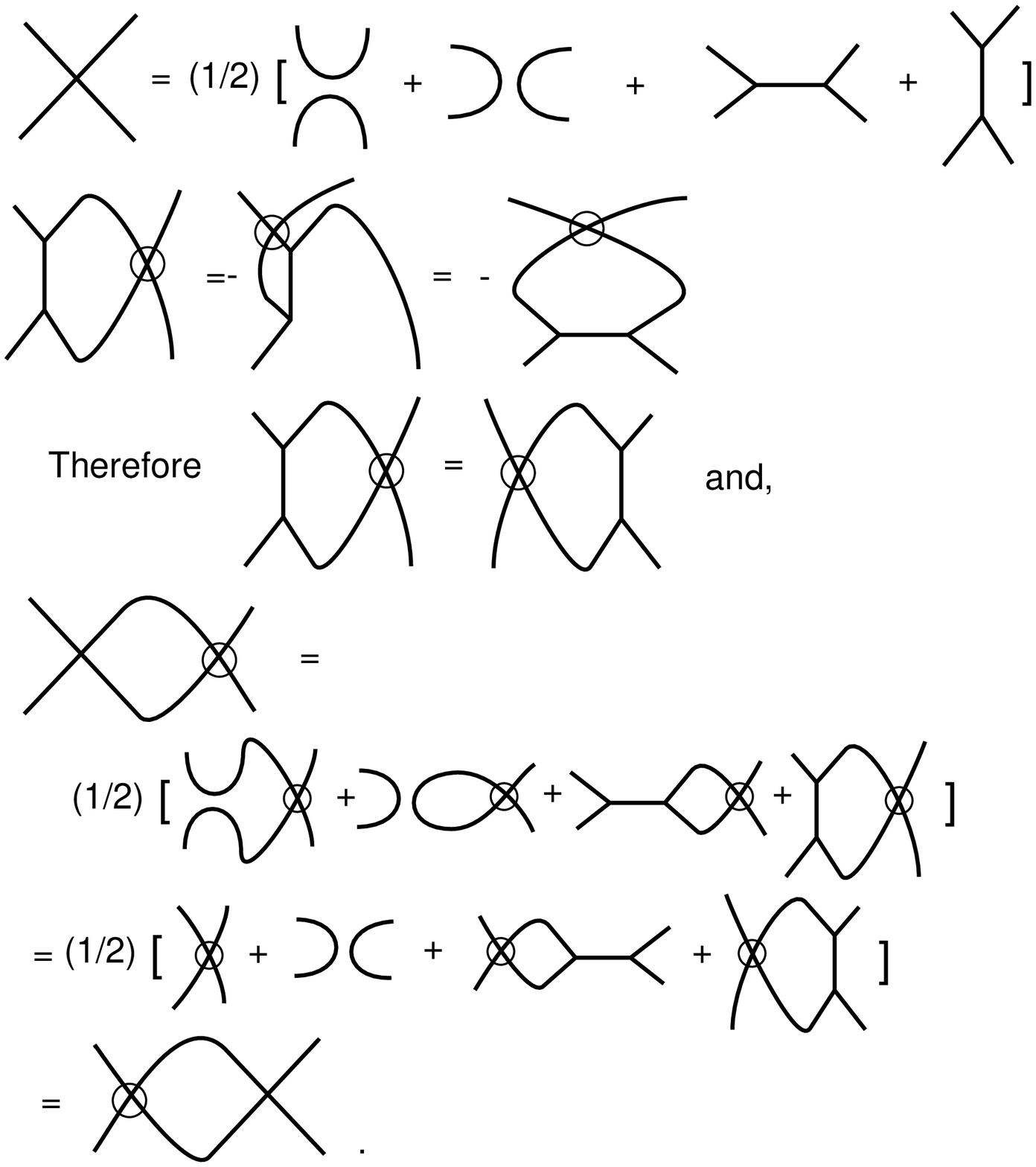}
     \end{tabular}
     \caption{\bf $G_2$ Evaluation is invariant under the $Z$-move.}
     \label{twist}
\end{center}
\end{figure}

\begin{figure}
     \begin{center}
     \begin{tabular}{c}
     \includegraphics[width=8cm]{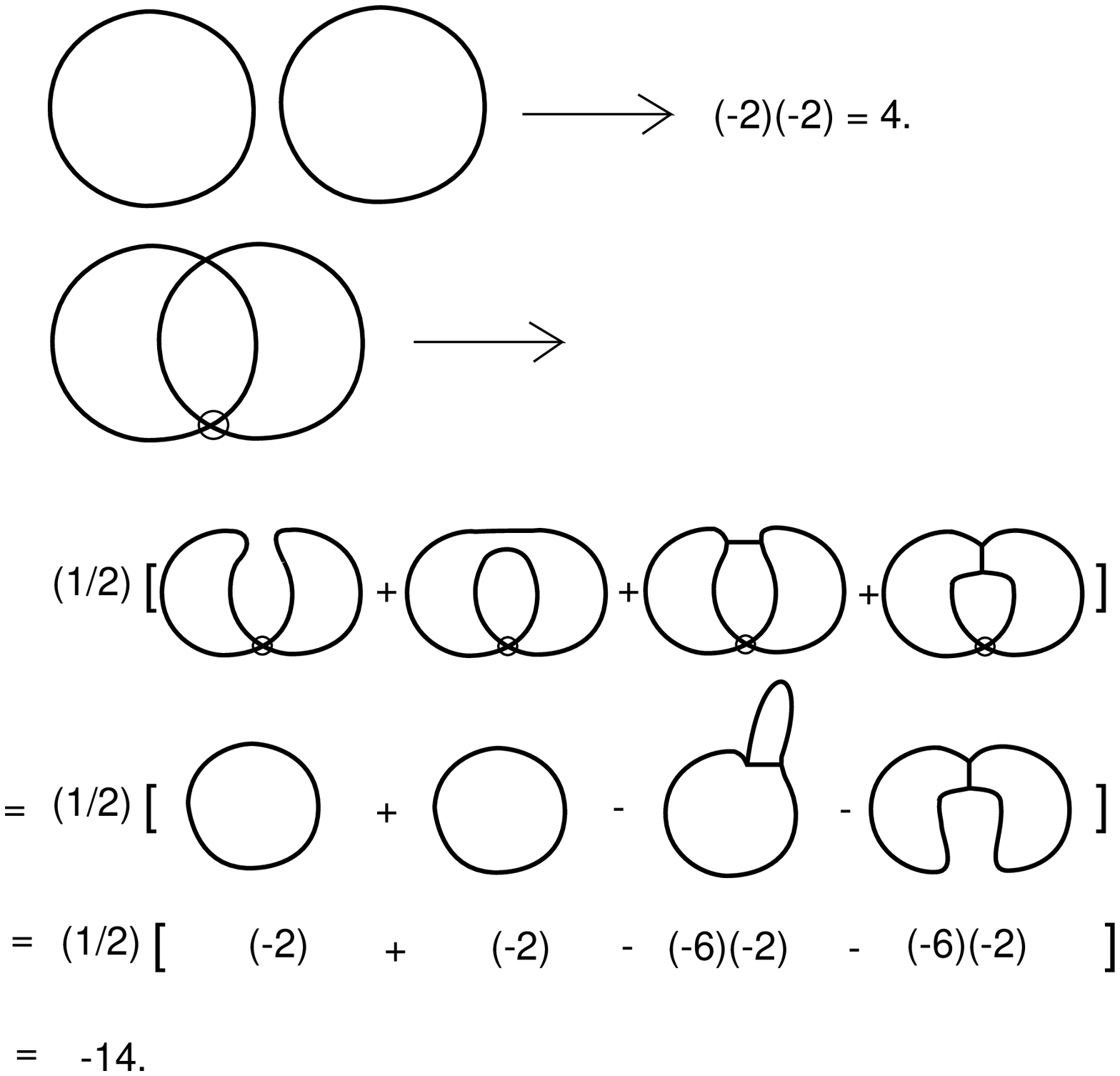}
     \end{tabular}
     \caption{\bf $G_2$ Evaluation of Simplest Free Link.}
     \label{flink}
\end{center}
\end{figure}

\begin{figure}
     \begin{center}
     \begin{tabular}{c}
     \includegraphics[width=5cm]{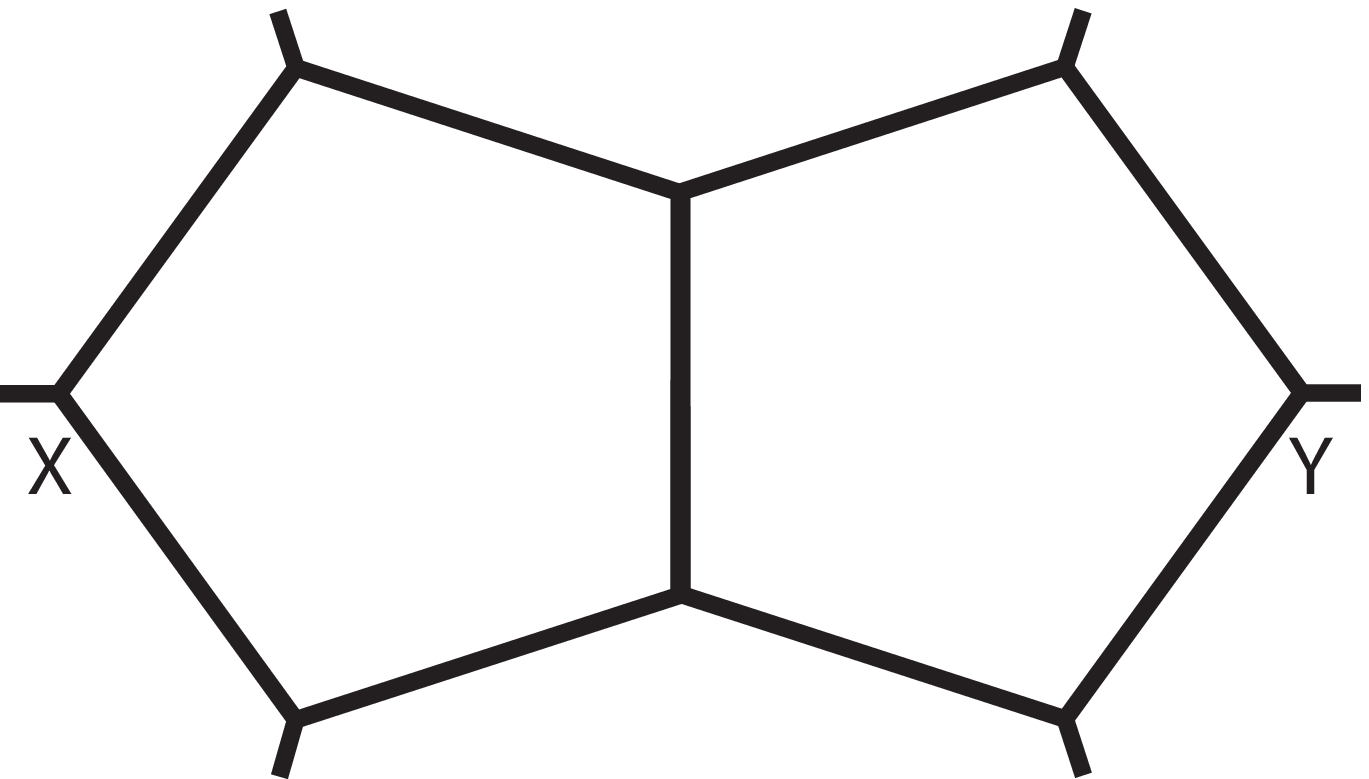}
     \end{tabular}
     \caption{\bf Five and Five Direct}
     \label{fivefivedir}
\end{center}
\end{figure}

\begin{figure}
     \begin{center}
     \begin{tabular}{c}
     \includegraphics[width=6cm]{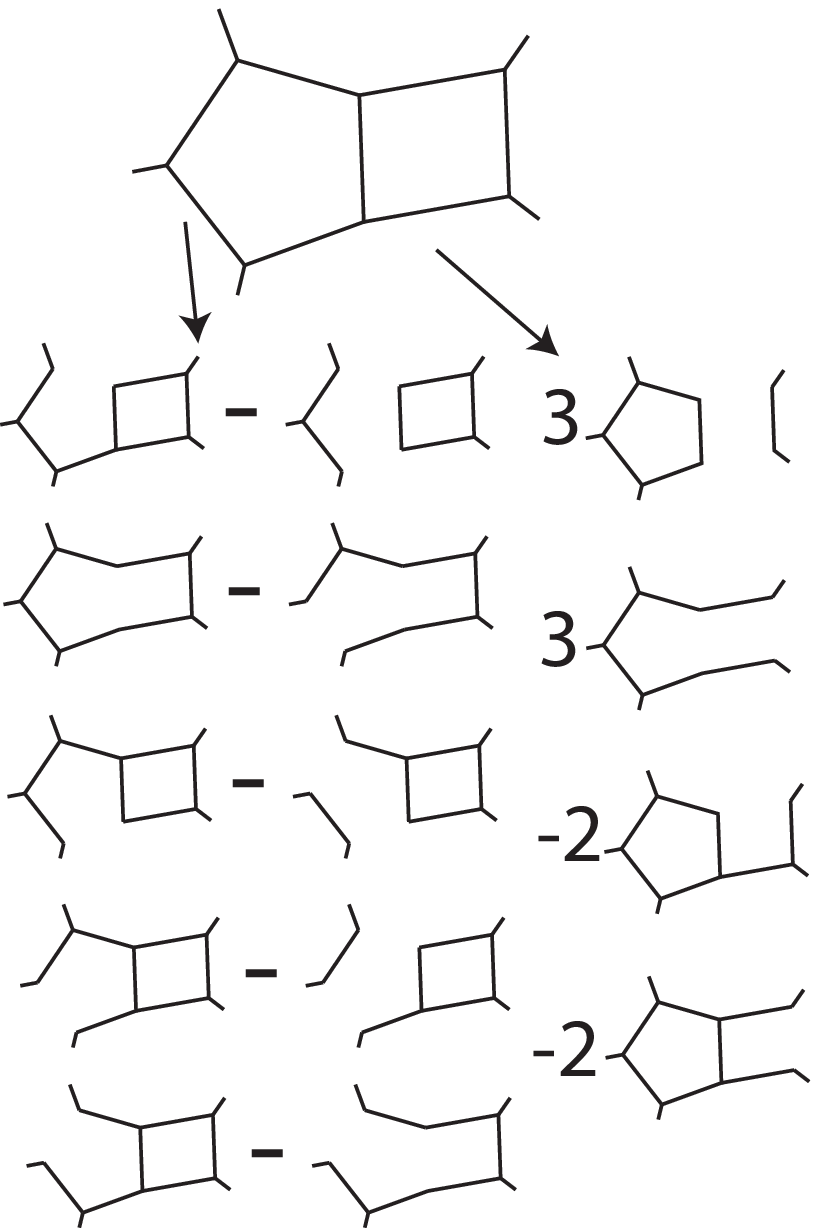}
     \end{tabular}
     \caption{\bf Five and Four Direct}
     \label{fivefour}
\end{center}
\end{figure}

\begin{figure}
     \begin{center}
     \begin{tabular}{c}
     \includegraphics[width=10cm]{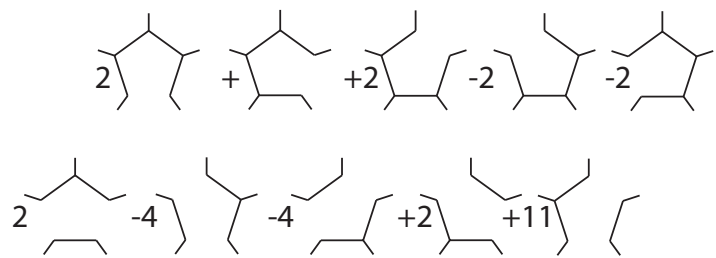}
     \end{tabular}
     \caption{\bf End Result of Five and Four}
     \label{fivefourend}
\end{center}
\end{figure}

\begin{figure}
     \begin{center}
     \begin{tabular}{c}
     \includegraphics[width=8cm]{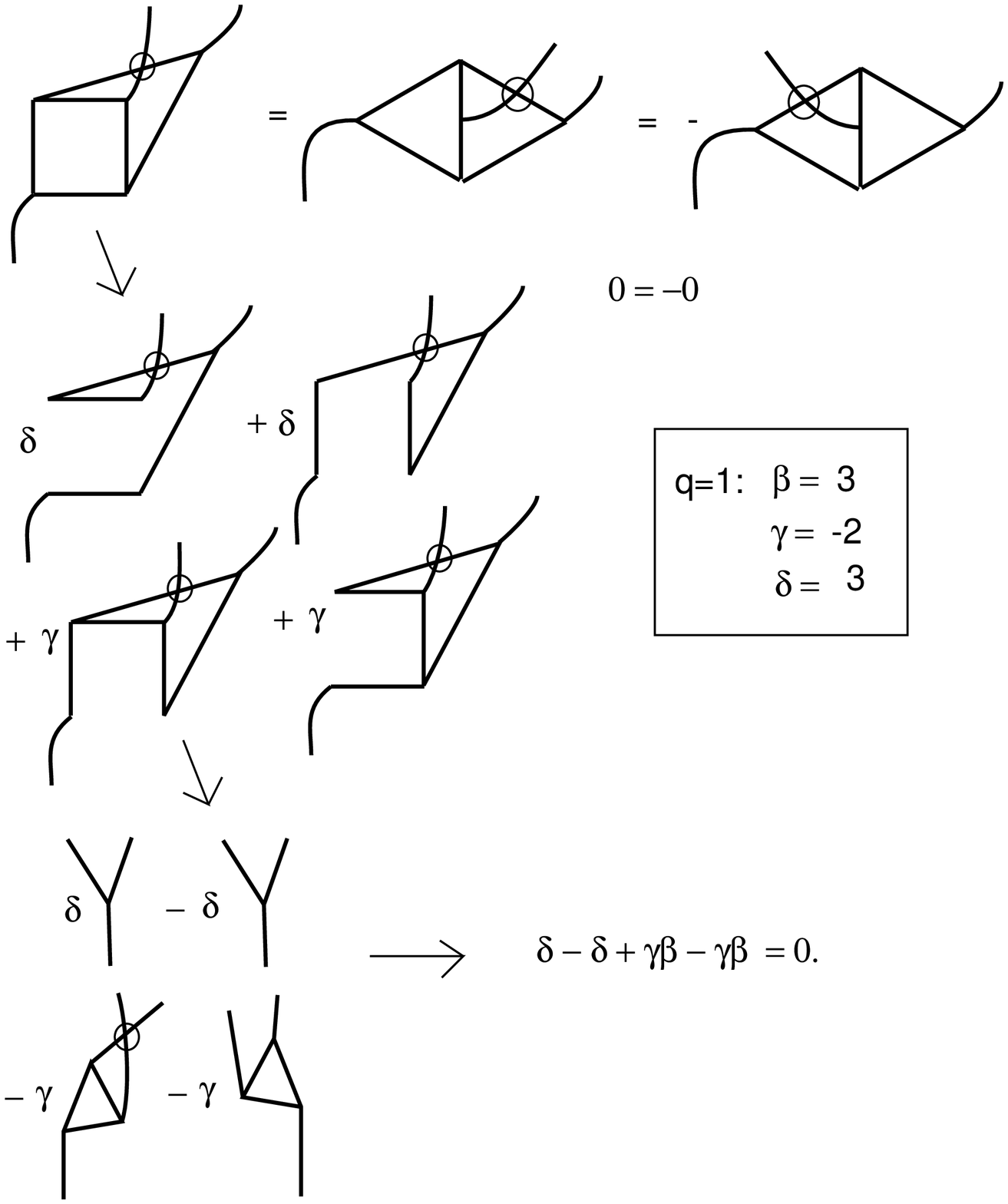}
     \end{tabular}
     \caption{\bf Four Interacting Virtually With Four}
     \label{fourvirtfour}
\end{center}
\end{figure}

\begin{figure}
     \begin{center}
     \begin{tabular}{c}
     \includegraphics[width=8cm]{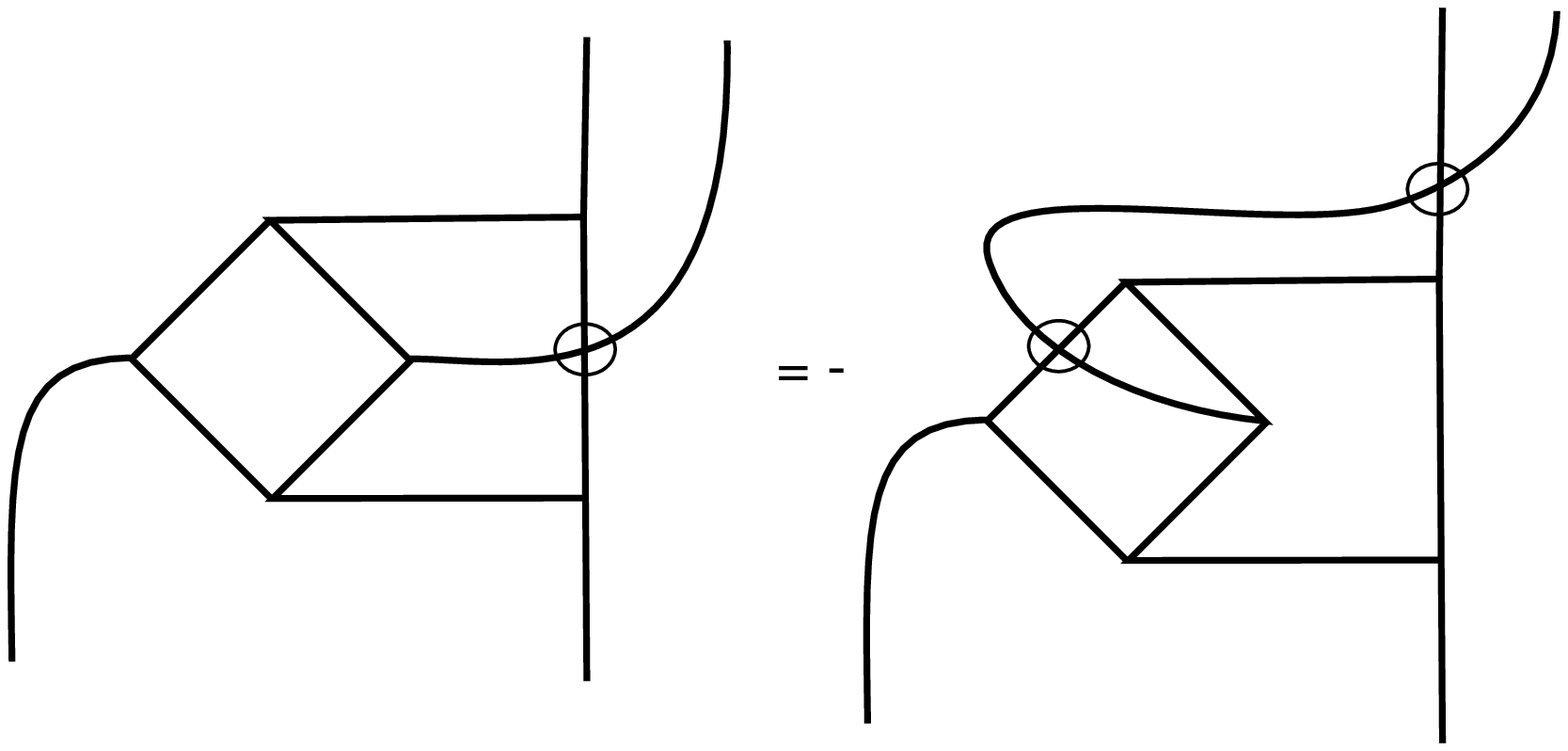}
     \end{tabular}
     \caption{\bf Four Interacting Virtually With Five}
     \label{fourvirtfive}
\end{center}
\end{figure}

\begin{figure}
     \begin{center}
     \begin{tabular}{c}
     \includegraphics[width=7cm]{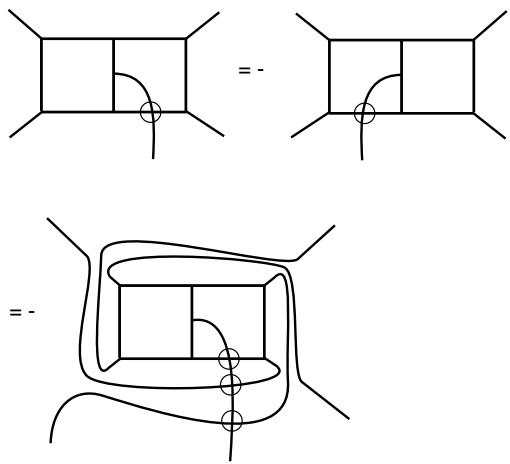}
     \end{tabular}
     \caption{\bf Five Interacting Virtually With Five Twist Identity}
     \label{fivevirtfive}
\end{center}
\end{figure}

\begin{figure}
     \begin{center}
     \begin{tabular}{c}
     \includegraphics[width=7cm]{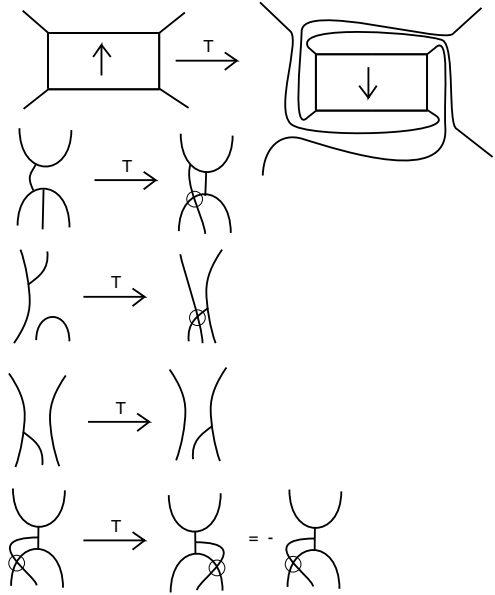}
     \end{tabular}
     \caption{\bf Five Interacting Virtually With Five Insertions}
     \label{fivevirtfiveinsert}
\end{center}
\end{figure}

\begin{figure}
     \begin{center}
     \begin{tabular}{c}
     \includegraphics[width=7cm]{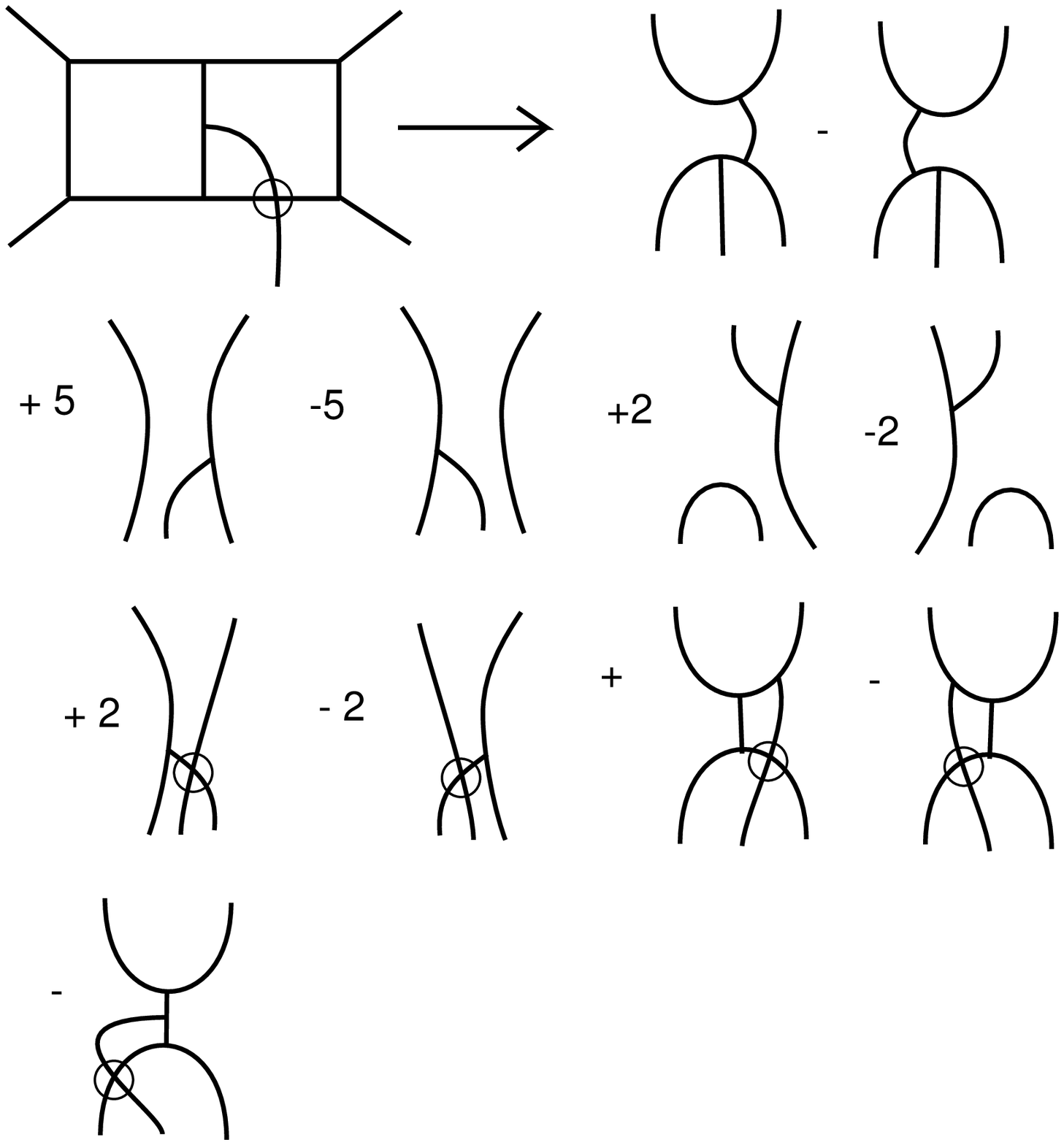}
     \end{tabular}
     \caption{\bf Five Interacting With Five Virtually - The  Expansion}
     \label{fivexp}
\end{center}
\end{figure}

\noindent{\bf Discussion.} For the $sl(3)$ invariant we have the advantage of using the orientation on the original free knot, which leads to the oriented states of the invariant and the restrictions to even numbers o sides to small polygons appearing in the expansion of the knot to the states. The $sl(3)$ invariant can be sometimes used to show non-invertibility of free knots. See \cite{Graph}.
\smallbreak

We are now in possession of the $G_2$ invariant for free knots and links.
We shall refer to it as the {\it free $G_2$ invariant} if necessary.
If a diagram has no bigons, triangles, quadrilaterals or pentagons, then no reduction is possible and the resulting free knot is non-trivial via this invariant. 
\smallbreak

The main advantage of the $so(3)$ invariant is that if there are no triangles in the expansion to graphs of the free knot, then one can reconstruct the free knot from that expansion. Examine Figure~\ref{kupdbl} and note that the $sl(3)$ invariant can, if there are no triangles in the skein expansion,  give us enough information to reconstruct a free knot. (We have referred to \cite{AC} for the theory of the $so(3)$ invariant. Its verification for free knots goes by the same lines as the verifications we have done here for $sl(3)$ and $G_2$.) 
\smallbreak

If there are triangles, then we can hope to use the $G_2$ invariant to make a finer discrimination.
We can combine the information already known from the $sl(3)$ and $so(3)$ invariants to narrow down the possibilities for applying the  $G_2$ invariant. Thus we can assume that we have a diagram that has triangles in its skein expansion. Such an unoriented diagram slips through the discrimination of both the 
$sl(3)$ and $so(3)$ invariants and can be tested with the $G_2$ invariant. We will discuss the construction of such examples in section 6.
\bigbreak

\clearpage

\section{Minimality Theorems and Uniqueness of Minimal Diagrams}

The existence of non-trivial free knots was first proved in $2009$ (See arxiv version of 
\cite{FirstFree}. See also \cite{Sbornik1}.). The proof relied upon the notion of {\em parity}:
the way of discriminating between {\em even} and {\em odd} crossings
of the knot diagram (respectively, chords of the chord diagram). The
simplest example of parity is {\em the Gaussian parity} as described in the third section of this paper. In the third section of this paper, we have discussed Corollary 1 showing that an irreducible odd diagram of a free knot $K$ is minimal in the strong sense that every diagram $K'$ equivalent to $K$ has
a smoothing identical with $K.$
\smallbreak

 Corollary $1$ of the previous section  is a strong statement: for
example, it follows from it that the minimal diagram is unique: if
$K$ is an $n$-crossing diagram, then no other $n$-crossing diagram
$K'$ can have $K$ as a smoothing since non-trivial smoothing leads
to a smaller number of classical crossings.
This result partially solves the recognition problem for free knots,
but it is solved for a very small class of them. Indeed, there
are lots of minimal diagrams of free knots which are minimal and
have even crossings.
\smallbreak
 
The result stated in Corollaries \ref{crl1} and \ref{crl2}
does not claim the uniqueness of the minimal diagram.
The main reason is that we can not restore the initial knot (or
link) $K$ from $K_{us}$ in a unique way. Indeed, given a trivalent
graph, in order to get a framed four-valent graph, we have to
contract some edges (those corresponding to unoriented smoothings of
vertices). This means that we have to find a matching for the set of
vertices of the whole trivalent graph $K_{us}$. But there might be
many matchings of this sort. After performing such a
matching, we still do not have a free knot diagram (framed
four-valent diagram), because we do not
know which incoming edge is opposite to which emanating edge.
Thus, there may be many diagrams $K$ having the same $K_{us}$. Of
course, in some cases, the whole invariant $Free[[K]]$ itself
does distinguish between them, nevertheless, sometimes it is not the
case.
View Figure~\ref{thirdunoriented}.
The bracket $K \mapsto Flat[[K]]$ is certainly invariant under any
Reidemeister moves; so is $K_{us}$ in the case of Corollary
\ref{crl1}.

\begin{figure}
\centering\includegraphics[width=200pt]{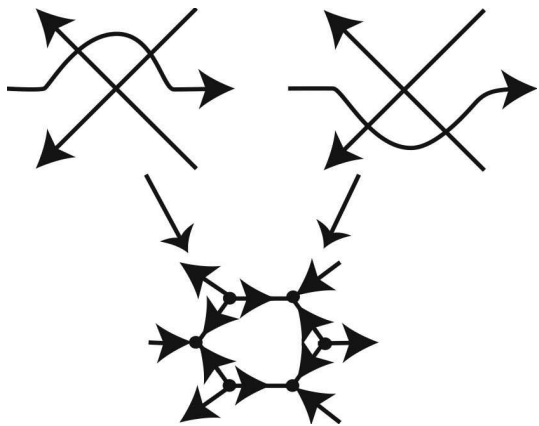}
\caption{\bf The third move does not change the bracket}
\label{thirdunoriented}
\end{figure}

Now, we see that in Figure~\ref{thirdunoriented}, two diagrams which
differ by a third Reidemeister move, lead to the same $K_{us}$.
However, we do not know whether there are different minimal knot
diagrams {\em of different knot types} sharing $K_{us}$ (or, even,
sharing $Free[[K]]$).
If the unoriented third Reidemeister move were the only case when
$Free[[K]]$ does not change then we would have not only a
criterion for minimality of a large class of diagrams, but also a
recognition algorithm for a large class of free knots.
\smallbreak

A free knot diagram (a framed four-valent graph) can not
contain loops or bigons.
If a free knot diagram $K$ contains a triangle, then it is
equivalent to a diagram $K'$ obtained from $K$ by a third
Reidemeister move.
If $K$ has girth greater than or equal to four, then no decreasing
Reidemeister move can be applied to $K$ and no third Reidemeister
move can be applied to $K$.
Thus we arrive at the following natural conjecture about free knots.

\begin{cj}
If a framed four-valent graph $K$ has girth greater than or equal to
$4$ then this diagram is the unique minimal diagram for the knot
class of $K$.
\end{cj}

In fact, the ``minimality'' part partially follows from Corollary
\ref{crl2}. Indeed, all diagrams of girth $5$ and many diagrams of
girth $4$ satisfy the conjecture. For a diagram of girth $4$ to
satisfy the conjecture, one should forbid the quadrilaterals of the
following sort (see Figure~ \ref{quadrispecialsort}) which lead to
quadrilaterals in $K_{us}.$

\begin{figure}
\centering\includegraphics[width=200pt]{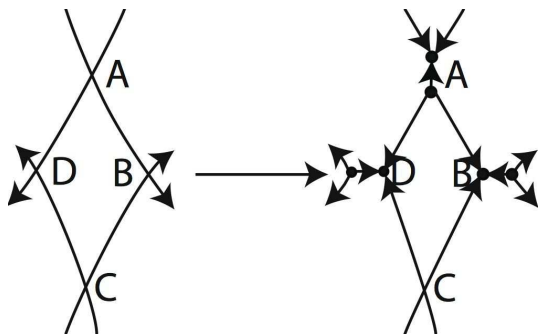}
\caption{\bf Quadrilaterals of a special sort in $K$ yield
quadrilaterals in $K_{us}$} \label{quadrispecialsort}
\end{figure}

For such a quadrilateral, there are two opposite vertices (say, $A$
and $C$) such that all the four edges of the quadrilateral emanate
from $A$ and $C$ (and come into the other two vertices, $B$ and
$D$).
Concerning the uniqueness of such minimal diagram, a very little is
known because of the ``matching ambiguity'' discussed above.
\smallbreak

\subsection{Minimality related to  $G_{2}$}
\begin{dfn}
Let $G$  be a flat (resp. free) trivalent graph.
We shall say that $G$ {\it contains} a graph $G'$ if $G'$ can be obtained from $G$ by a squence of the Kuperberg reduction rules for the flat(resp. free) $G_{2}$ invariant.
\end{dfn}

Let $K$ be a flat (resp. free) link diagram.
Note that all those graphs $K_{s}$ corresponding to various leading states
$s$ of $K$, appear in the expansion of $\{\{K\}\}_{Free}$ (resp., $\{\{K\}\}_{Flat}$)
with the coefficient $1/2^{n}.$ Some of these graphs are reducible, and some of them are not. 
However, no irreducible leading states will cancel in the state sum.
This leads to the following

\begin{thm}
Assume $K$ be a flat (resp. free)  $4$-regular graph, and let $S$ be a leading state of $K$
such that $K_{s}$ is irreducible in the flat (resp., free) category.
Then every graph $K'$ equivalent to $K$ as a flat (resp., free) link 
 has a state that contains $K_{s}.$

In particular, $G$ is minimal, non-classical and non-trivial.
\label{thmin}
\end{thm}

Note that one can construct free knot diagrams satisfying the conditions of Theorem \ref{thmin}
with all even crossings, having some triangles and not satisfying minimality conditions
for $sl(3)$ and $sl(3)$ graphical invariants for free knots. For example, this can be achieved by
using chord diagrams.

\section{Examples}
This section gives examples related to the discussions in the paper. 
The first example we consider is given in Figure~\ref{octalink}. Here we have an example of a flat virtual 
link diagram $L.$ In Figure~\ref{octastate} we illustrate a state of this link in the expansion of the flat $sl(3)$ bracket $\{\{L\}\}_{Flat}.$ We call this the ``octahedral state" of $K.$  In 
Figure~\ref{octasurface} we show how the state (and hence the corresponding link) embeds in a surface with boundary. We now analyze this state graph $G$ and show that it is irreducible in the $sl(3)$ bracket expansion of $K,$ thus proving that $K$ is a non-trivial flat link.
\bigbreak

Examine the surface $S$ in Figure~\ref{octasurface} and it is apparent that it has four boundary components with the following numbers of edges from the graph for the corresponding faces:
$8,8,16,16.$ Thus there are no small faces in this graph and so it is not possible to make any reductions of the graph in the setting of the $sl(3)$ invariant. It is not hard to see that this state receives a non-zero coefficient in the invariant, and hence the link $L$ is a non-trivial flat link. Incidentally, if $F$ is the surface obtained from $S$ by adding a disc to each boundary component, and $g$ is the genus of $F,$ then we have $v = 16, e = 24, f = 4$ where $v,e,f$ denote the number of vertices, edges and faces of the corresponding decomposition of $F.$ Then $2-2g = v - e + f = 16 -24 + 4 = -4$ and hence $g = 3.$ The surface $F$ has genus equal to three. Note that this means that the link $L$ cannot be represented on any surface of genus smaller than $3$ since the irreducible state graph $G$ must appear in every 
such representation. Since $L$ can be represented in genus three, this proves that $L$ has surface genus three as a flat link. 
\bigbreak

The example shown in Figure~\ref{hard} may be a non-trivial free knot, but it is not detected by the invariants discussed in this paper.
\bigbreak

\begin{figure}
     \begin{center}
     \begin{tabular}{c}
     \includegraphics[width=8cm]{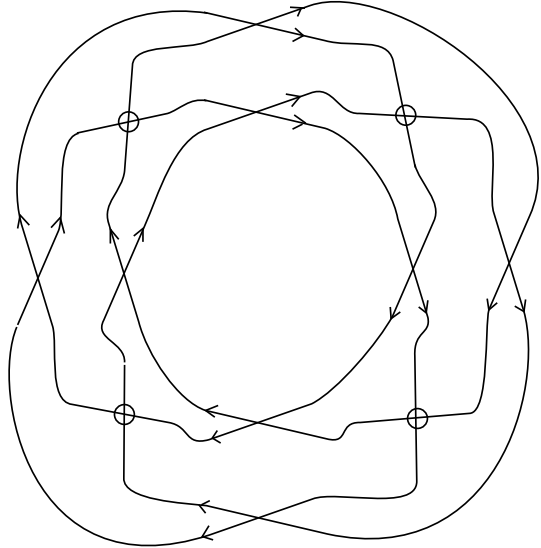}
     \end{tabular}
     \caption{\bf The flat link L}
     \label{octalink}
\end{center}
\end{figure}

\begin{figure}
     \begin{center}
     \begin{tabular}{c}
     \includegraphics[width=8cm]{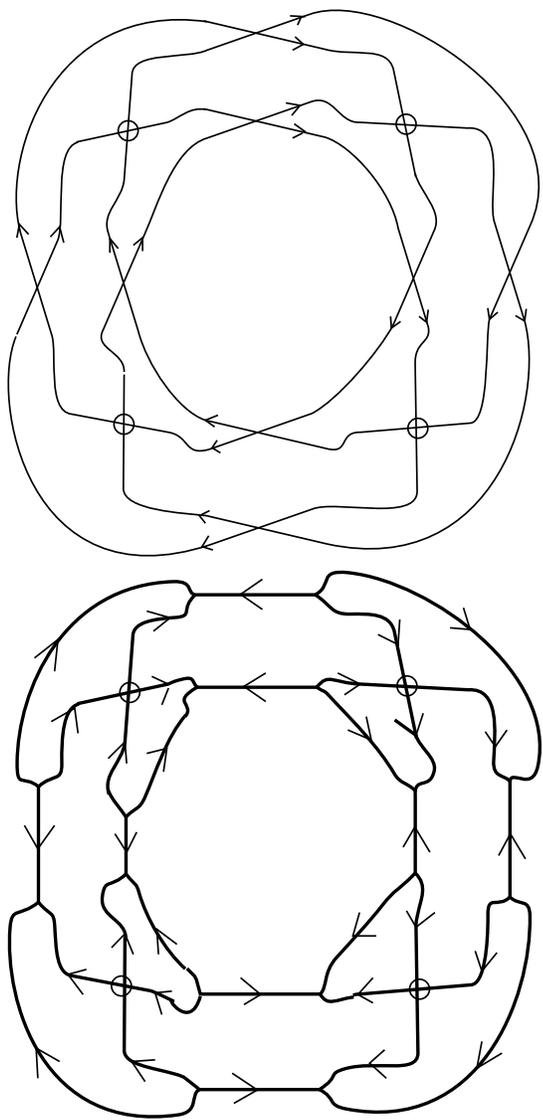}
     \end{tabular}
     \caption{\bf A state of the flat link L}
     \label{octastate}
\end{center}
\end{figure}

\begin{figure}
     \begin{center}
     \begin{tabular}{c}
     \includegraphics[width=8cm]{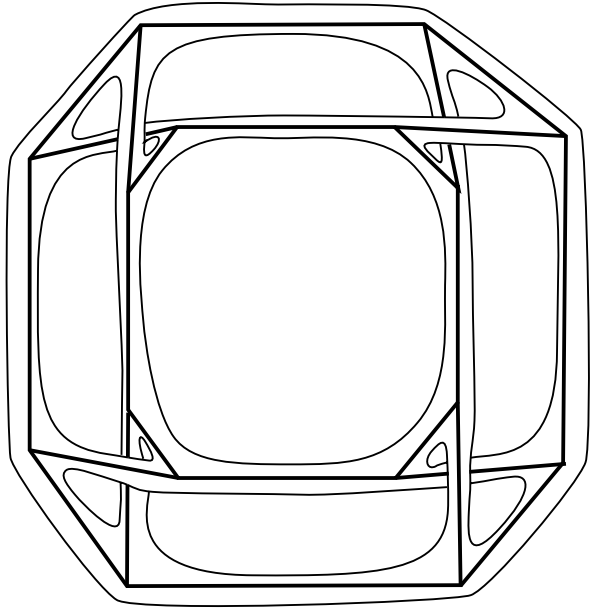}
     \end{tabular}
     \caption{\bf A surface for the state of the flat link L}
     \label{octasurface}
\end{center}
\end{figure}

\begin{figure}
     \begin{center}
   \begin{tabular}{c}
    \includegraphics[width=8cm]{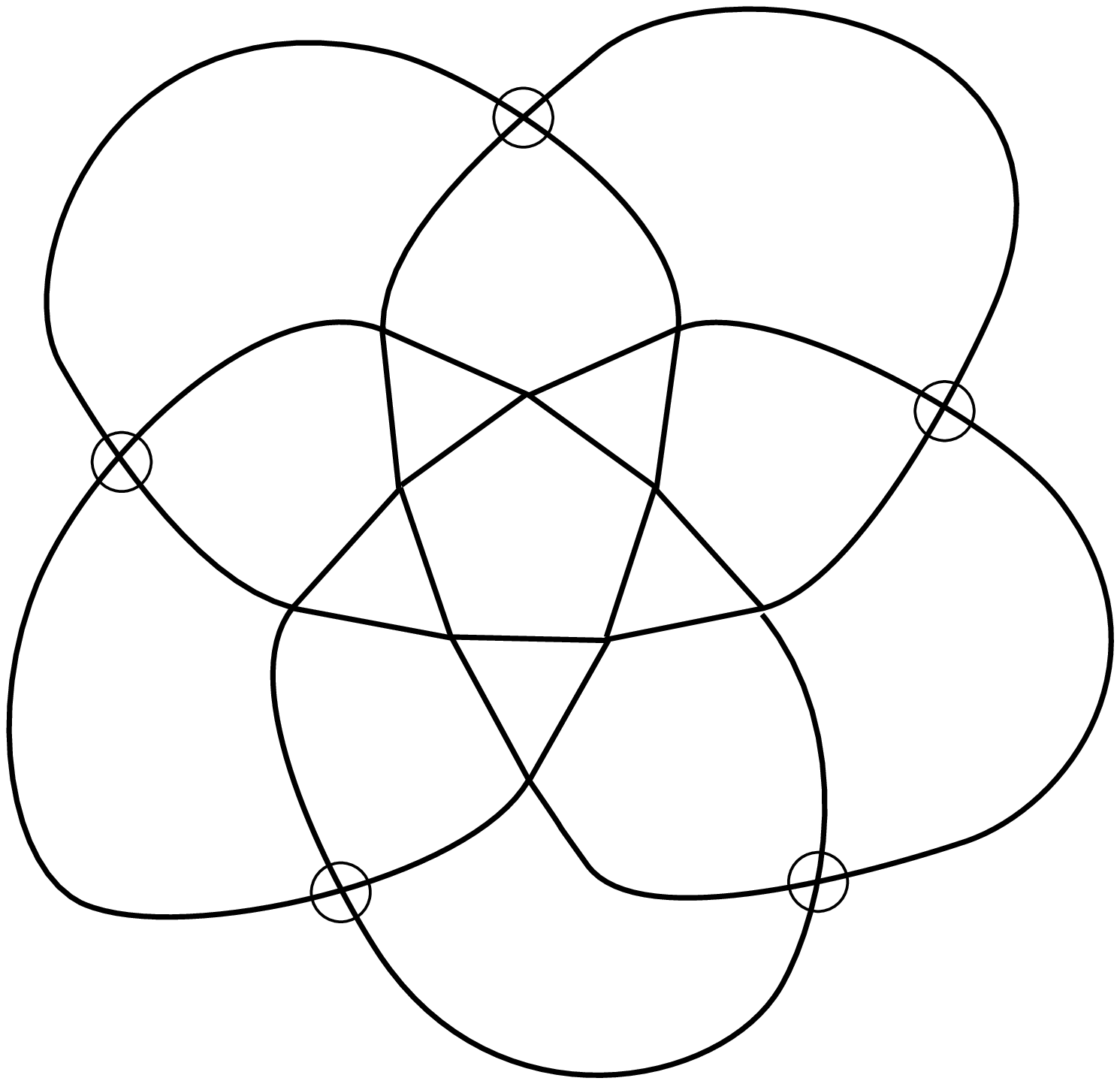}
     \end{tabular}
     \caption{\bf Possibly Irreducible but Undectectable by $G_2$ Invariant}
    \label{hard}
\end{center}
\end{figure}

The next example, depends  on Figure~\ref{wtick} and Figure~\ref{whirl}.
In Figure~\ref{whirl} we show a free knot $G$ where the dark nodes are the classical vertices of this free knot and the simple crossings are the virtual crossings (since there are many virtual crossings we deemed this change of convention necessary). We show that the $G_{2}$ bracket detects the non-triviality of $G$ by exhibiting a state that survives in the invariant and is irreducible. The state is constructed in analogy with the state construction shown in Figure~\ref{wtick}. In Figure~\ref{wtick}
we show a typical pair of adjacent classical vertices and we show the corresponding $G_2$ state.
We illustrate this in Figure~\ref{wtick} for a circle of $3$ vertices and ask the reader to consider the same pattern for the $17$ vertices of $G$ in Figure~\ref{whirl}. It is then apparent to inspection that $G$ has no cycles (just as a graph with no restrictions on framing at crossings) of small length ($2,3,4$ or $5$) and that the same is true of the corresponding state $S$ of $G.$ Thus $S$ is irreducible and it is easy to see that $S$ appears with non-zero coefficient in the $G_2$ invariant of $G,$ since any irreducible sate with a maximal number of paired trivalent replacements has positive coefficient in the invariant. This completes the proof that $G$ is non-trivial.
\bigbreak

We should further remark that there is a special state that can be constructed from a free knot $L.$
This state is the one where crossings are replaced by pairs of trivalent vertices such that the resulting state graph is bipartite. Call this state $B(L).$ The reader will be interested to find examples of free knots that are distinguished by their bipartite state.
\bigbreak

\begin{figure}
     \begin{center}
     \begin{tabular}{c}
     \includegraphics[width=8cm]{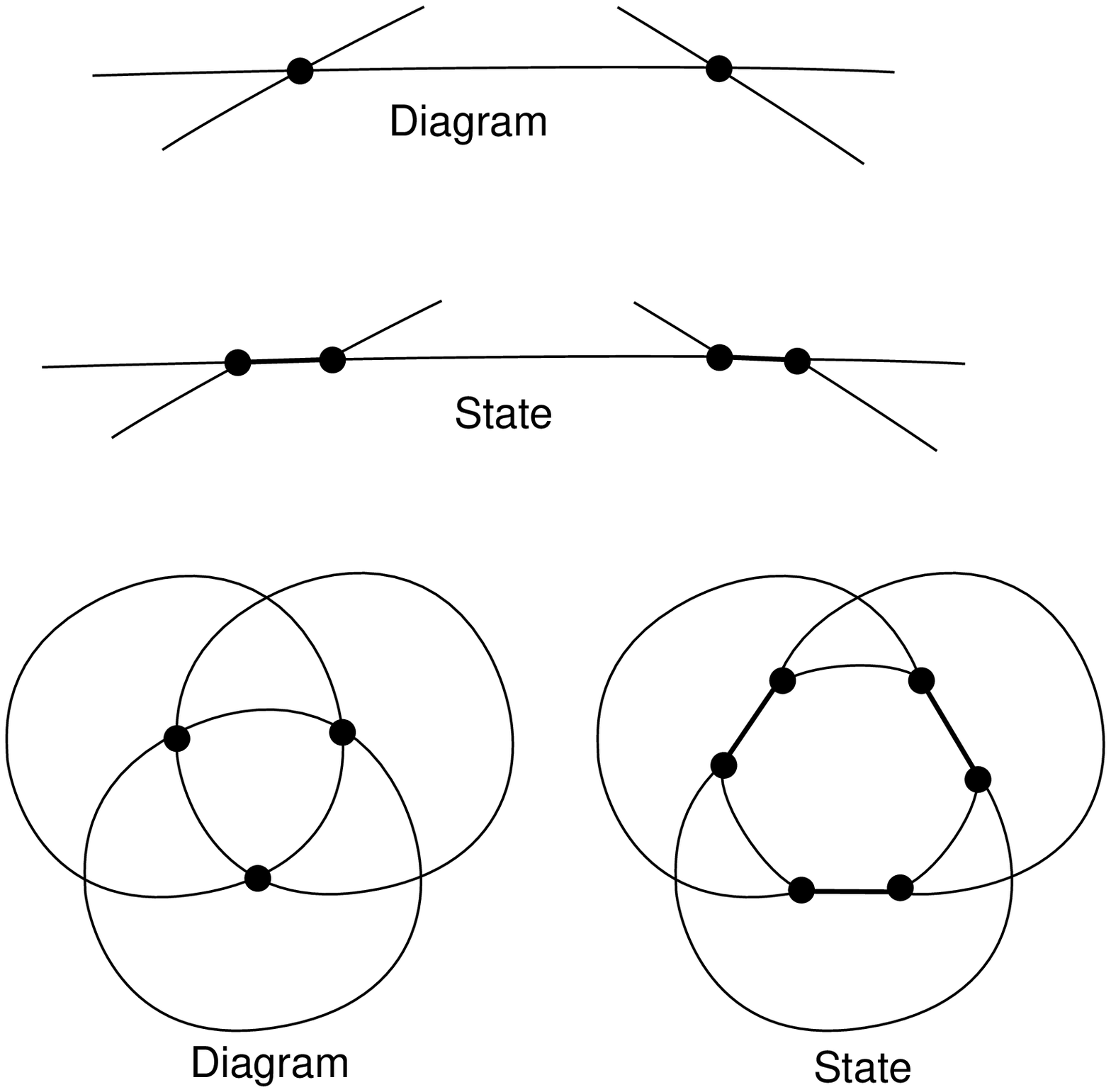}
     \end{tabular}
     \caption{\bf Diagram and State}
     \label{wtick}
\end{center}
\end{figure}

\begin{figure}
     \begin{center}
     \begin{tabular}{c}
     \includegraphics[width=8cm]{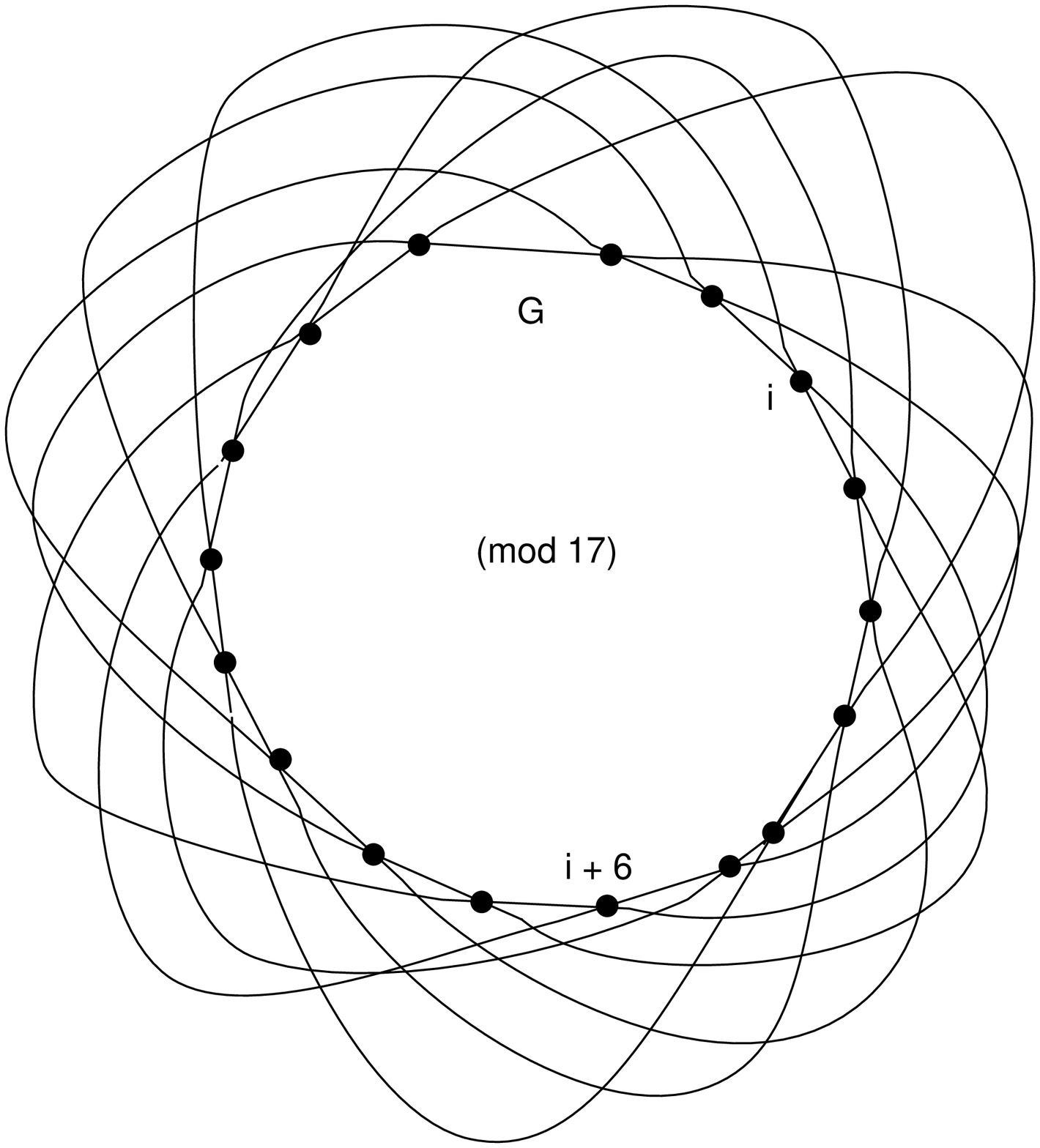}
     \end{tabular}
     \caption{\bf A Non-trivial Free Knot G}
     \label{whirl}
\end{center}
\end{figure}

\bigbreak

\section{On the Penrose Coloring Bracket}

A {\it proper edge $3$-coloring} of a trivalent graph is an
assignment of three colors (say from the set $\{ 0,1,2 \}$) to the
edges of the graph such that three distinct colors appear at every
vertex. In \cite{Penrose} R.Penrose gives a bracket that counts the
proper edge $3$-colorings of a trivalent plane graph. \smallbreak

\begin{figure}
     \begin{center}
     \begin{tabular}{c}
     \includegraphics[width=8cm]{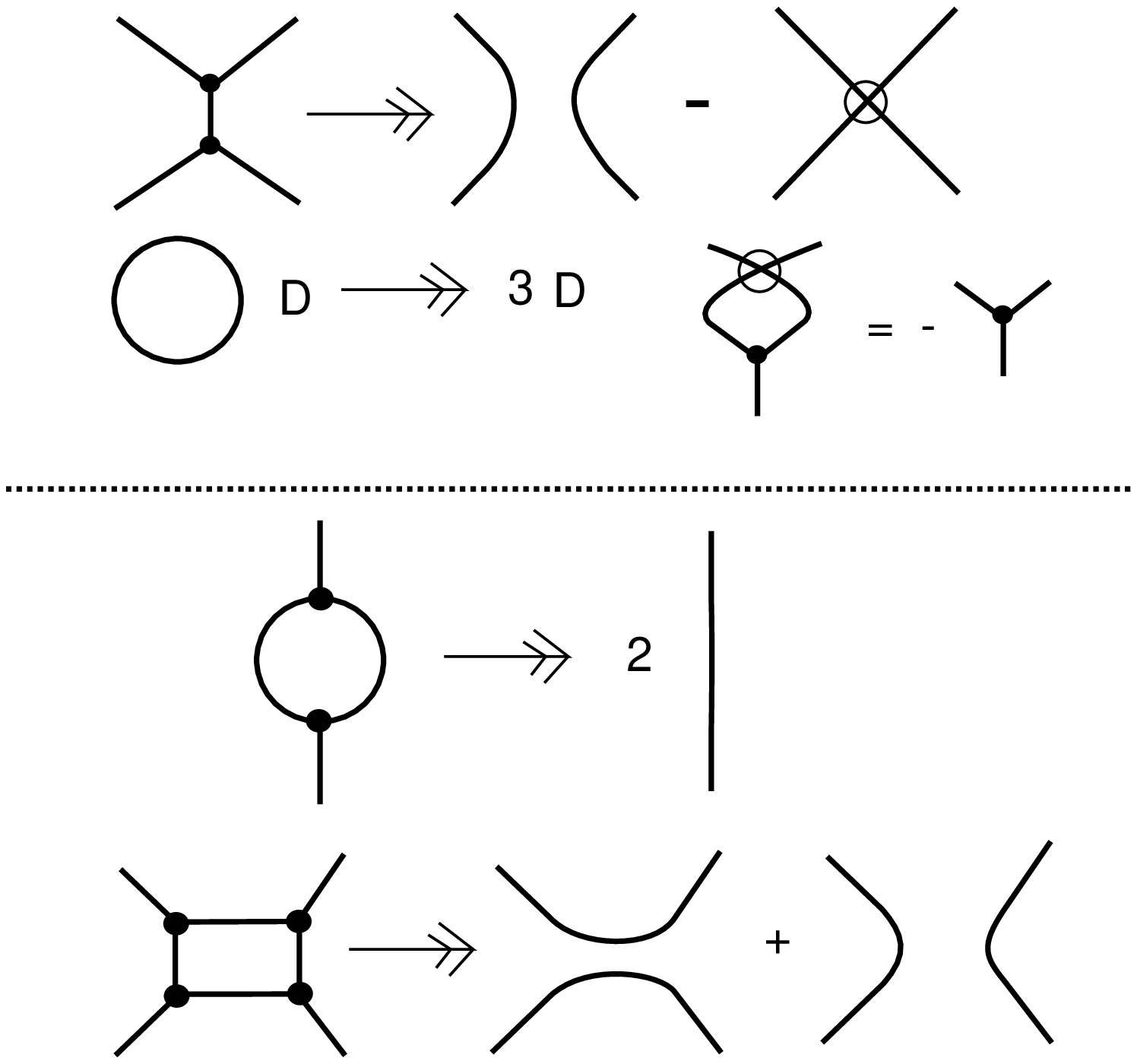}
     \end{tabular}
     \caption{\bf Penrose Bracket and Identities}
     \label{penrose}
\end{center}
\end{figure}

\begin{figure}
     \begin{center}
     \begin{tabular}{c}
     \includegraphics[width=8cm]{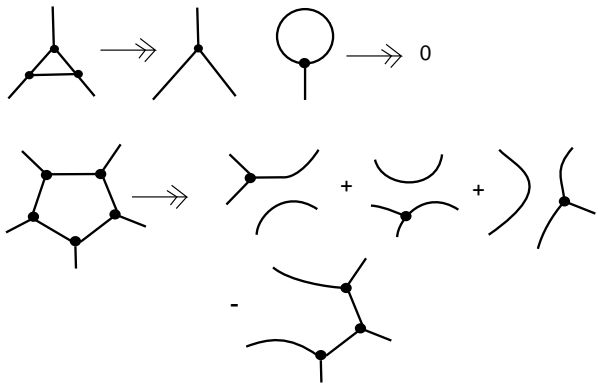}
     \end{tabular}
     \caption{\bf Further Penrose Bracket Identities}
     \label{penrose2}
\end{center}
\end{figure}

\begin{figure}
     \begin{center}
     \begin{tabular}{c}
     \includegraphics[width=8cm]{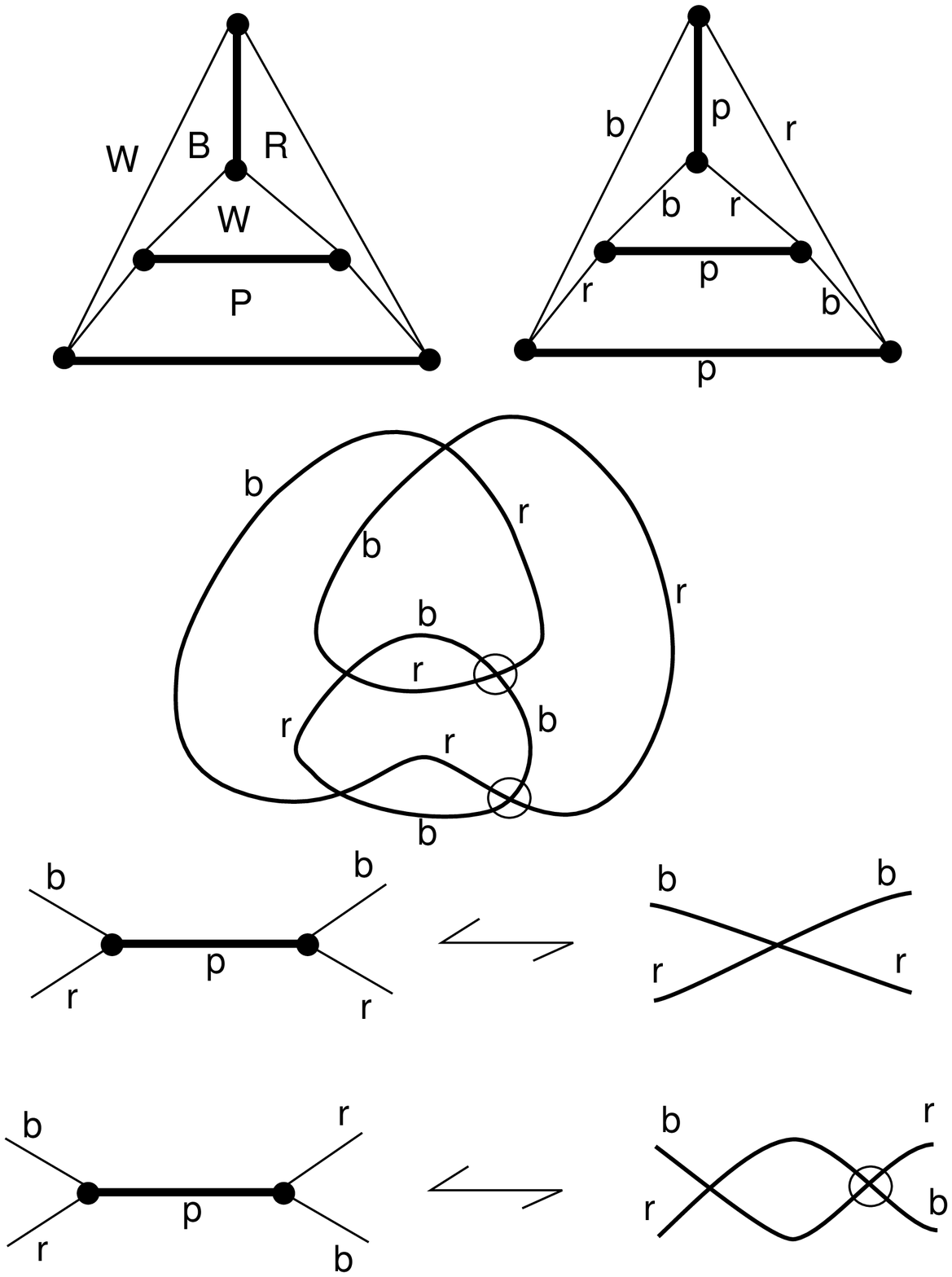}
     \end{tabular}
     \caption{\bf Coloring Graphs and Free Knots}
     \label{penrose3}
\end{center}
\end{figure}

The Penrose bracket \cite{Penrose} is defined as shown in
Figure~\ref{penrose}. Penrose devised this bracket and other
combinatorial expansions based on abstract tensors. In the case of
this Penrose bracket, we will show that it is a special case of the
Kuperberg $sl(3)$ bracket and explain how flat virtul knots are
related to classical graph coloring problems. \bigbreak

The Penrose bracket is an evaluation of trivalent graphs immersed in
the plane so that cyclic orders at the nodes are given. It can be
defined independently of the planar immersion so long as the cyclic
orders at the nodes are specified, but counts colorings correctly
only for planar {\it embeddings.} Thus the Penrose bracket is
well-defined for virtual graphs (See Definition~\ref{vgraphs}.) as we have defined them earlier in
the paper. Note that in the expansion formula for the Penrose
bracket we have one term with crossed lines and a virtual crossing.
In completely expanding a graph by this formula one has many curves
with virtual self-crossings. Each curve is evaluated as the number
$3$ and thus the evaluation of the Penrose bracket is a signed sum
of powers of $3.$ \smallbreak

Note in  Figure~\ref{penrose} the identities that we have listed.
These follow from the the expansion formulas above the line. It is
then apparent that the evaluation of the Kuperberg $sl(3)$ bracket
at $A=1$ on an oriented trivalent plane graph is equal to the value
of the Penrose bracket for this graph. The Kuperberg bracket is
calculated by the same reduction formulas as the Penrose bracket,
and for the oriented graph, we only need the bigon and quadrilateral
reductions. This observation is the main reason we mention the
Penrose bracket in this context. \smallbreak

In Figure~\ref{penrose2}  we show the other reduction identities for
the Penrose bracket (for triangles, pendant loops and five-sided
regions). These identites can be used to evaluate the Penrose
bracket for unoriented plane graphs, since an Euler characteristic
argument tells us that any trivalent graph in the plane with no
isthmus must have a region with less then six sides. We can take the
reduction formalism of the Penrose bracket and attempt to use it to
create new invariants of free knots. This will be the subject of a
separate paper. \smallbreak

Finally, we point out a relationship between coloring maps, graphs
and free knots and links. Examine Figure~\ref{penrose3}. At the top
left of the figure we show a trivalent graph in the plane with a
four-coloring of its regions by the colors $\{W,R,B,P\}.$ Two
regions that share an edge are colored with different colors. This
graph is a simplest example of a plane graph that requires four
colors. Adjacent to the coloring of the regions of this graph we
have illustrated a coloring of its edges by the three colors
$\{r,b,p\}$ so that every node of the graph has three distinct
colored edges incident to it. The edge coloring can be obtained from
the face coloring by regarding $\{W,R,B,P\}$ as isomorphic to the
group $\Z_{2} \times \Z_{2}$ with $W$ the identity element and $R^2
= B ^2 = P^2 = W$ and $RB = BR = P, RP = PR = B, BP = PB = R.$ Then
the edge coloring is obtained by coloring each edge with the product
of the region colors on either side of it. \smallbreak

\begin{dfn}
Call a component  in a free link diagram (represented as a planar
diagram) {\it even} if there are an even number of virtual crossings
on this component that occur between the given component and other
components of the link. (See Section 2.3 for definitions related to
framed and free knots and links.) Call a free link diagram {\it
componentwise  even} if every component of the link is even. Note
that a single component free knot diagram is componentwise even
since there are zero virtual crossings between the one component and
any other component.
\end{dfn}

In the right hand side in Figure~\ref{penrose3} we illustrate a free
link diagram  that is obtained from the edge-colored graph above it
by using the translations illustrated  below the free link diagram.
In these translations we use two-colors for the edges of a free
link. {\it Opposite edges are colored by different colors from the
binary set $\{ r,b \}$. }

\begin{dfn}
We call a free link diagram  {\em $2$-colored} if it is colored in
$\{ r,b\}$ so that opposite edges have different colors. It is clear
from the definition that {\it a free link diagram can be $2$-colored
if and only if it is componentwise even.}
\end{dfn}

By using the translations in the figure we see that {\it an
edge-colored trivalent graph immersed in the plane corresponds
uniquely to a colored free link diagram.} In the fiigure we see that
a colored planar graph may correspond to a free link diagram that
has virtual crossings. The relationship between planar colorings and
properties of free links needs further study. \smallbreak

\section{Braids and the Kuperberg  $sl(3)$ Bracket}

\begin{figure}
     \begin{center}
     \begin{tabular}{c}
     \includegraphics[width=8cm]{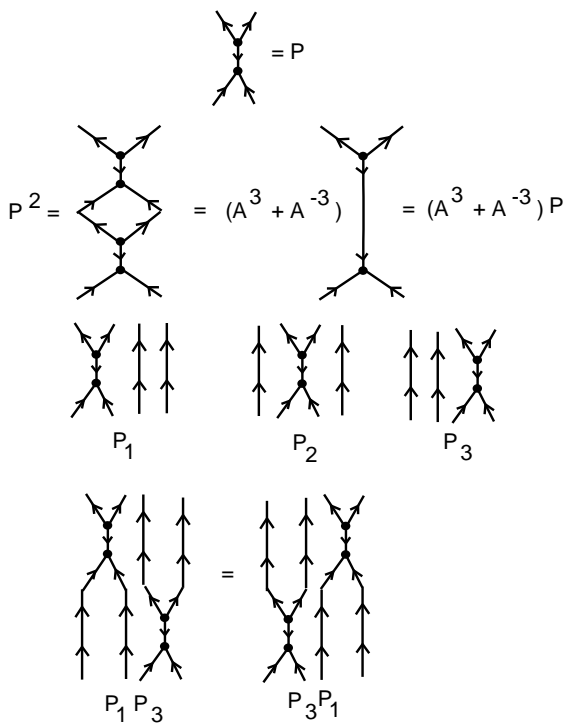}
     \end{tabular}
     \caption{\bf Hecke Algebra 1}
     \label{hecke1}
\end{center}
\end{figure}

\begin{figure}
     \begin{center}
     \begin{tabular}{c}
     \includegraphics[width=8cm]{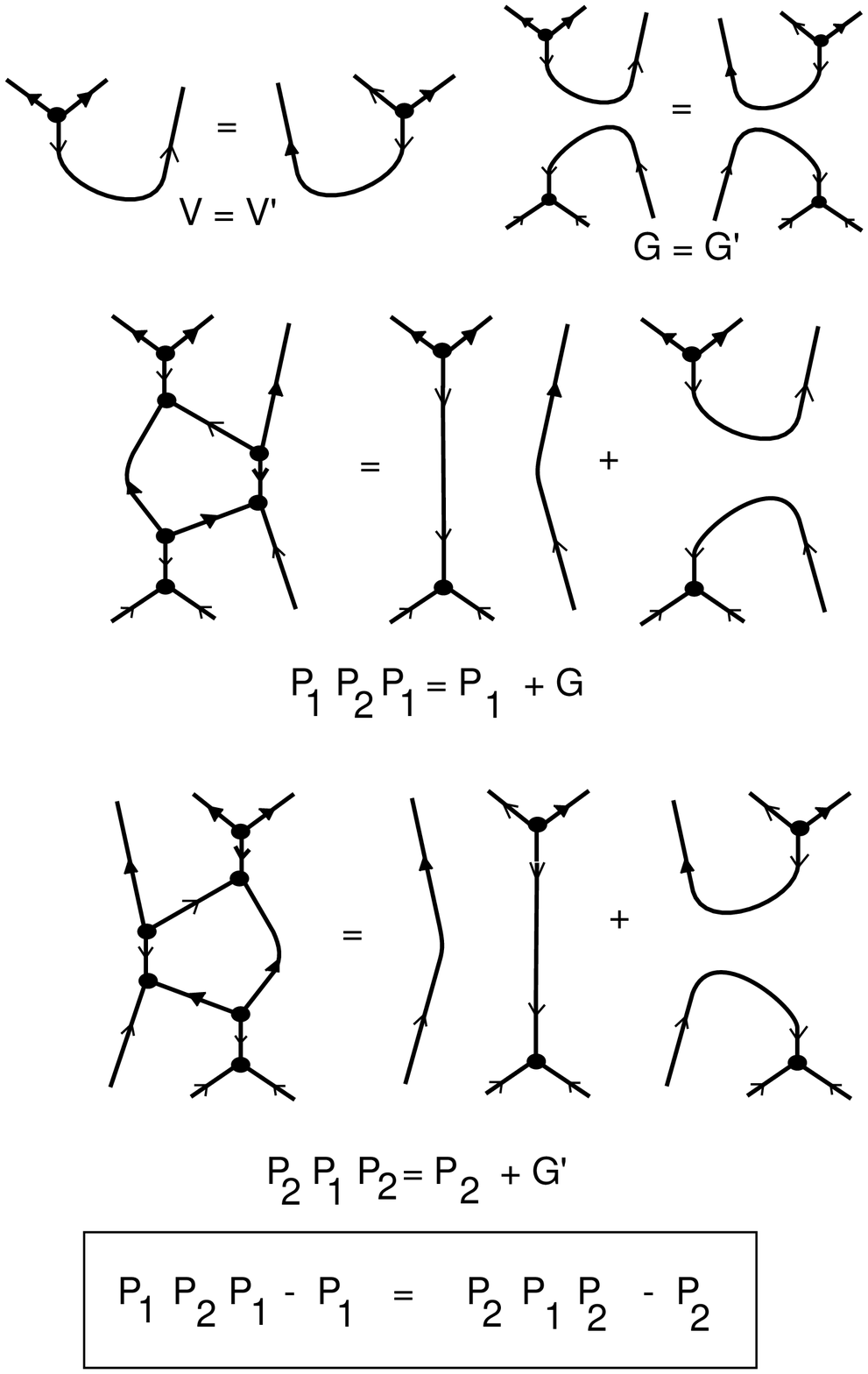}
     \end{tabular}
     \caption{\bf Hecke Algebra 2}
     \label{hecke2}
\end{center}
\end{figure}

\begin{figure}
     \begin{center}
     \begin{tabular}{c}
     \includegraphics[width=8cm]{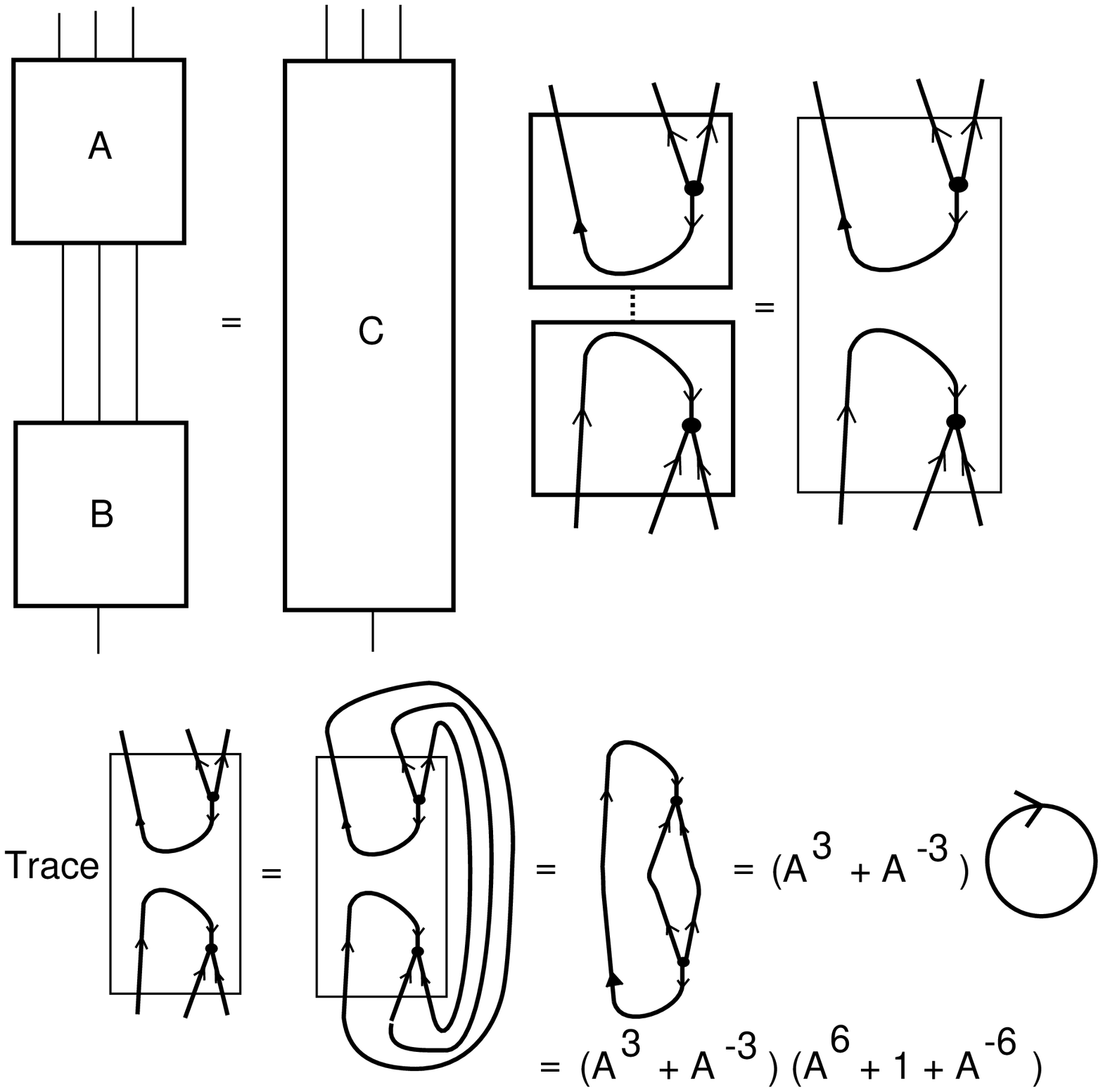}
     \end{tabular}
     \caption{\bf Graph Category and Trace}
     \label{graphcat}
\end{center}
\end{figure}

Graphs can act on graphs by systmatically attaching edges of one graph to the edges of another.
Braid groups and representation of them into graphical cateegories can be defined in this manner.
All the invariants considered in this paper can be defined for
braids with values in a special algebra of tangles with polyomial or
numerical coefficients. In this section we concentrate on the
formulation of the $sl(3)$ invariant for braids.  \bigbreak

We will use our graphical reduction method to give a
graph-polynomial valued trace function on the Virtual Hecke algebra.
Our procedure is analogous to the graphical trace on the
Temperley-Lieb algebra \cite{KP} that is implicit in the bracket
model for the Jones polynomial. Here we apply the same ideas to our
graphical analysis of the Kuperberg $sl(3)$ bracket for virtual
knots. In order to accomplish this, we first recall the formalism of
the Kuperberg  $sl(3)$ bracket as described in \cite{Ohtsuki} and
translate this into braid form. View Figure~\ref{kupbra}. Here we
show the expansion formulas for the Kuperberg  $sl(3)$ bracket. The
reader can regard the expansion of a crossing as the expansion of a
generator $\sigma_{i}$  of the Artin braid group. Then we see that
this expansion has the form
$$\sigma_{i} = A^{2} I  - A^{-1} P_{i}$$ where $P_{i}$ stands for the double-$Y$-graph depicted in
Figure~\ref{hecke1}. In that Figure we have shown how $P_{i}$ is
formed for $i=1,2,3$ in the image of the braid group on $4$ strands.
We have also illustrated how $$P_{i}^{2} = (A^{3} + A^{-3})P_{i}$$
and in Figure~\ref{hecke2} we have illustrated how the quadrilateral
expansion formula for the Kuperberg  $sl(3)$ bracket leads to the
graphical identities:
\begin{enumerate}
\item $P_{i}P_{i+1}P_{i} = P_{i} + G_{i}$
\item  $P_{i+1}P_{i}P_{i+1} = P_{i+1} + G'_{i}$
\item  $G_{i} =  G'_{i}$
\end{enumerate}
Thus one can conclude the graphical identity
$$P_{i}P_{i+1}P_{i} - P_{i} = P_{i+1}P_{i}P_{i+1} - P_{i+1}. $$
and it is easy to see that
$$P_{i}P_{j} = P_{j}P_{i}$$ for $|i-j|>2.$
\smallbreak

\begin{dfn}
We formalize this multiplication of graphs via the pattern shown in
Figure~\ref{graphcat}. We make a category, $GraphCat,$ by using
graphs with free ends (nodes of order one) immersed in the plane.
Other than the free ends, the graphs have  trivalent vertices
(oriented as in this paper with all edges pointing either in to a
vertex or out from the vertex). They may have virtual crossings in
their planar immersions. We include the circle among these graphs,
but it is reduced to a scalar just as earlier in this paper. The
ends of the graphs are arranged so that each graph is in a rectangle
with its sides parallel to the standard vertical and horizontal
directions. The free ends of the graph are either at the top of the
box or at the bottom and they are oriented compatibly with the
graph. A box with no free ends from the top or from the bottom is
given a dotted line at the top or bottom just for notational
purposes. The objects in $GraphCat$  consist in the set $Obj = \{
[0],[1],[2],\cdots \}$ in $1-1$ correspondence with the non-negative
integers. A graph $G$ with $m$ free lower ends and $n$ free upper
ends is a morphism
$$G: [m] \longrightarrow [n].$$
The most general morphisms in $GraphCat$ are linear combinations of
graph-morphisms from $[n]$ to $[m]$ ($n$ and $m$ are fixed for a
given linear combination)  with coefficients in the ring
$\Z[A,A^{-1}]$. Products are defined by taking the set of
graph-morphisms as generators for a module over the ring
$\Z[A,A^{-1}]$ and formally multiplying elements of this ring.
Individual graph morphisms multiply by their categorical composition
as described above.
\end{dfn}
\smallbreak

The graph in the box is taken up to equivalences using the Kuperberg
$sl(3)$ bracket expansion. This means that bigons and quadrilaterals
are expanded out and any circles that result are evaluated.

\begin{dfn}
A {\em morphism} in this category is the equivalence class of such a
box.
\end{dfn}

Note that this equivalence class will often be a linear combination
of graph classes. Morphisms are composed by direct multiplication
(attachment of bottom ends to top ends as in the Figure) of
graph-boxes, and the distributive law applies to the linear
combinations. Note that two graph-boxes can be multiplied only if
they have a matching number of free edges. We call this category
$GraphCat.$ \smallbreak

Recall that ${\cal M}$ is the module
$\Z[A,A^{-1}][t_{1},t_{2},\cdots]$ of Laurent polynomials over
formal commutative products of graphs from $\cal T$ where $\cal T$
denotes the set of equivalence classes of graphs without free ends.
We define $$Trace:GraphCat \longrightarrow {\cal M}$$ by the formula
$$Trace(G) = [[\bar{G}]]$$ where $\bar{G}$ denotes the closure of $G$ directly analogous to the standard closure of a braid, as shown in Figure~\ref{graphcat} and $[ [\bar{G}]]$ denotes the reduced equivalence class in  ${\cal M}$ of the linear graph combination $\bar{G}.$
\smallbreak

From the point of view of the theory of braids the Hecke algebra
$H_{n}(q)$ is a quotient of the group ring $\Z[q,q^{-1}][B_{n}]$ of
the Artin braid group by the ideal generated by the quadratic
expressions
\begin{equation} \label{HeckeQuad}
\sigma_{i}^2 - z\sigma_{i} - 1
\end{equation}
for $i = 1, 2, \ldots n-1,$ where $z = q - q^{-1}.$ This corresponds
to the identity $\sigma_{i} - \sigma_{i}^{-1} = z1$, which is
sometimes regarded diagrammatically as a  skein identity for
calculating knot polynomials. By the same token, we define the {\it
virtual Hecke algebra} $VH_{n}(q)$ to be the quotient of the group
ring of the virtual braid group $VB_{n}$ (See \cite{VirtualC})
${\Bbb Z}[q,q^{-1}][VB_{n}]$ by the ideal generated by
Eqs.~(\ref{HeckeQuad}). Here the reader can verify that with $q =
A^{2}$ and the algebraic expressions ($I$ is the identity element in
the algebra),
\begin{enumerate}
\item $\sigma_{i} = A^{2} I  - A^{-1} P_{i},$
\item $\sigma^{-1}_{i} = A^{-2} I  - A^{1} P_{i},$
\item$ P_{i}^{2} = (A^{3} + A^{-3})P_{i},$
\item $P_{i}P_{i+1}P_{i} - P_{i} = P_{i+1}P_{i}P_{i+1} - P_{i+1}.$
\item $P_{i}P_{j} = P_{j}P_{i}$ for $|i-j|>2.$
\end{enumerate}

Using these definitions we  arrive at the identical Hecke algebra
that is obtained from the equation~(\ref{HeckeQuad}). The equations
involving the elements $P_{i}$ ensure that the braiding relations
are satisfied by the images of the $\sigma_{i}$ in the Hecke
algebra. Of course there are virtual generators $v_{i}$ in the
virtual braid group and these satisfy relations that we shall omit
to write here. The reader can consult \cite{VirtualC} for this
information. The virtual Hecke algebra also has virtual elements,
that we call $v_{i}.$ The virtual element $v_{i}$ in $VB_{n}$
consists in $i-1$ parallel strands, then a virtual crossing of
strands $i$ and $i+1$ and then parallel strands from strand $i+2$ to
strand $n.$ We will assume that $v_{i}P_{i} = P_{i}v_{i} = P_{i}$
for this representation. This is in accord with our previous
assumptions about  the behaviour of virtual crossings that are
adjacent to a trivalent graphical vertex. We let $VH_{n} =
VH_{n}(q)$ denote the virtual Hecke algebra on $n$ strands, written
in the above form using $P_{i}$. In this form, we have a surjection
$$h:VB_{n} \longrightarrow VH_{n}$$ described by these equations. Note that we use $\sigma_{i}$ both for the braid group generator and for its image in the Hecke algebra. In the Hecke algebra  $\sigma_{i}$ expands in terms of the identity and the projector $P_{i}.$
\smallbreak

Note that in all diagrams involving the virtual Hecke algebra or
involving virtual braids, when we speak of $VB_{n}$ or $VH_{n}$, the
number of free ends at the bottom is equal to the number of free
ends at the top and is equal to $n$ in both cases. \smallbreak

Along with this mapping of the virtual braid group to the virtual
Hecke algebra we have a functor from the category of virtual braids
to the category of graphs $GraphCat$, $$g:B_{n} \longrightarrow
GraphCat,$$ given by expanding each braid generator according to the
Kuperberg  $sl(3)$ bracket as in Figure~\ref{kupbra}.  It is clear
at once from the discussion above that this map $g$ factors through
the virtual Hecke algebra so that we have
$$k: VH_{n} \longrightarrow GraphCat$$ taking the virtual Hecke algebra as a category on one object to a subalgebra of $GraphCat$ and so that $g = k \circ h.$ This means that for a virtual braid $b$ with standard closure $\bar{b},$ the extended Kuperberg
$sl(3)$ bracket with values in $M$ can be expressed by the formula
below where $k$ is the map defined above with domain the virtual
Hecke algebra.
$$[[\bar{b}]]= Trace(g(b)) = Trace(k(h(b))$$  Here
$Trace: GraphCat \longrightarrow {\cal M}$ is the map defined at the
beginning of this section. \smallbreak

The extended Kuperberg  $sl(3)$ bracket gives a topologically
invariant function on the virtual braid group, and it gives a trace
function on the virtual Hecke algebra. This is the first time that
such a trace function has been constructed for the virtual Hecke
algebra. For the virtual Hecke algebra we define $Tr:VH_{n}
\longrightarrow M$ by the formula
$$Tr(x) = Trace(k(x)).$$  In other words, we interpret basic products in the virtual Hecke algebra as
graphs in $GraphCat$ and we evaluate them by closure and the
reduction via Kuperberg  $sl(3)$ bracket relations. Note that
$Tr(XY) = Tr(YX)$ for $X,Y \in VH_{n}$  because both $XY$ and $YX$
have the same closure.  Other properties of this trace will be
analyzed in a separate publication. We have shown here that our
generalized Kuperberg  $sl(3)$ bracket can be expressed in terms of
this trace on the virtual Hecke algebra. \smallbreak

\section  {\bf Remarks}
This article arose through our discussions of new possibilities in
virtual knot theory and in relation to advances of Manturov using
parity in virtual knot theory, particularly in the area of free
knots. In the paper \cite{MOY}, there is a model for the
$sl(n)$-version of the Homflypt polynomial for classical knots. This
model is  based on patterns of smoothings as shown in
Figure~\ref{muhoya}.

\begin{figure}
     \begin{center}
     \begin{tabular}{c}
      \includegraphics[width=100pt]{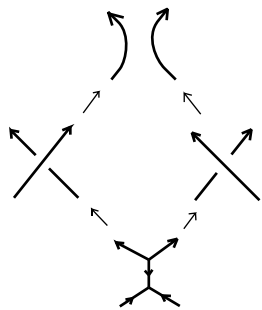}
     \end{tabular}
     \caption{\bf Murakami-Ohtsuki-Yamada relation for $sl(n)$}
     \label{muhoya}
\end{center}
\end{figure}

These patterns suggested to us the techniques we use in this paper
with the Kuperberg brackets, and we expect to generalize them
further. In this method,  the value of the polynomial for a knot is
equal to the linear combination of the values for two graphs
obtained from the knot by resolving the two crossings as shown in
Figure~\ref{muhoya}.
 \bigbreak

\end{document}